\documentclass[a4paper]{amsproc}
\usepackage{amssymb}
\usepackage{mathrsfs}
\usepackage{amsfonts}
\usepackage{amsthm}
\usepackage{amsmath,amscd}
\usepackage{amsxtra}     % Use various AMS packages
\usepackage{bm}
\usepackage[all,cmtip]{xy}
\usepackage{epsfig}
\usepackage{verbatim}
\usepackage{color}
\usepackage{enumerate}
\usepackage{hyperref}
\usepackage{tikz}
\usetikzlibrary{arrows,matrix,shapes,trees}

\theoremstyle{plain}
 \newtheorem{thm}{Theorem}[section]
 \newtheorem{prop}{Proposition}[section]
 \newtheorem{lem}{Lemma}[section]
 \newtheorem{cor}{Corollary}[section]
\theoremstyle{definition}
 \newtheorem{ex}{Example}[section]
 \newtheorem{defn}{Definition}[section]
\theoremstyle{remark}
 \newtheorem{rmk}{Remark}[section]
 \numberwithin{equation}{section}

\newcommand{\N}{{\mathbb N}}
\newcommand{\Q}{{\mathbb Q}}
\newcommand{\Z}{{\mathbb Z}}

\newcommand{\Gm}{\mathbb{G}_{\mr{m}}}

\newcommand{\Gml}{\mathbb{G}_{\mr{m,log}}}
\newcommand{\Gmlb}{\overline{\mathbb{G}}_{\mr{m,log}}}

\newcommand{\mr}{\mathrm}
\newcommand{\mc}{\mathcal}

\newcommand{\fsS}{(\mr{fs}/S)}
\newcommand{\fsSet}{(\mr{fs}/S)_{\mr{\acute{e}t}}}
\newcommand{\fsSfl}{(\mr{fs}/S)_{\mr{fl}}}
\newcommand{\fsSket}{(\mr{fs}/S)_{\mr{k\acute{e}t}}}
\newcommand{\fsSkfl}{(\mr{fs}/S)_{\mr{kfl}}}
\newcommand{\Set}{S_{\mr{\acute{e}t}}}
\newcommand{\Sket}{S_{\mr{k\acute{e}t}}}
\newcommand{\Sfl}{S_{\mr{fl}}}
\newcommand{\Skfl}{S_{\mr{kfl}}}
\newcommand{\Spec}{\mathop{\mr{Spec}}}

\title[Extending tamely ramified strict 1-motives into k\'et log 1-motives]{Extending tamely ramified strict 1-motives into k\'et log 1-motives}



\address{
    Heer Zhao, 
    Fakult\"at f\"ur Mathematik, 
    Universit\"at Duisburg-Essen, 
    Essen 45117, 
    Germany, 
    	heer.zhao@uni-due.de}

\begin{document}

%{\begin{flushleft}\baselineskip9pt\scriptsize
%PUBLICATIONS DE L'INSTITUT MATH\'EMATIQUE\newline
%Nouvelle s\'erie, tome 91(105) (2012), od--do \hfill DOI:
%\end{flushleft}}
\vspace{18mm} \setcounter{page}{1} \thispagestyle{empty}

\begin{abstract}
We define k\'et abelian schemes, k\'et 1-motives, and k\'et log 1-motives, and formulate duality theory for these objects. Then we show that tamely ramified strict 1-motives over a complete discrete valuation field can be extended uniquely to k\'et log 1-motives over the corresponding discrete valuation ring. As an application, we present a proof to a result of Kato stated in \cite[\S 4.3]{kat4} without proof.
\end{abstract}

\maketitle
\setcounter{tocdepth}{2}
\tableofcontents

\section*{Notation and conventions}
Let $S$ be an fs log scheme. We denote by $(\mr{fs}/S)$ the category of fs log schemes over $S$, and denote by $(\mr{fs}/S)_{\mr{\acute{e}t}}$ (resp. $(\mr{fs}/S)_{\mr{k\acute{e}t}}$, resp. $(\mr{fs}/S)_{\mr{fl}}$, resp. $(\mr{fs}/S)_{\mr{kfl}}$) the classical \'etale site (resp. Kummer \'etale site, resp. classical flat site, resp. Kummer flat site) on $(\mr{fs}/S)$. In order to shorten formulas, we will mostly abbreviate $\fsSet$ (resp. $\fsSket$, resp. $\fsSfl$, resp. $\fsSkfl$) as $\Set$ (resp. $\Sket$, resp. $\Sfl$, resp. $\Skfl$). We refer to \cite[2.5]{ill1} for the classical \'etale site and the Kummer \'etale site, and \cite[Def. 2.3]{kat2} and \cite[\S 2.1]{niz1} for the Kummer flat site. The definition of the classical flat site is an obvious analogue of that of the classical \'etale site. Then we have two natural ``forgetful'' maps of sites:
\begin{equation}\label{eq0.1}
\varepsilon_{\mr{\acute{e}t}}:(\mr{fs}/S)_{\mr{k\acute{e}t}}\rightarrow (\mr{fs}/S)_{\mr{\acute{e}t}}
\end{equation}
and
\begin{equation}\label{eq0.2}
\varepsilon_{\mr{fl}}:(\mr{fs}/S)_{\mr{kfl}}\rightarrow (\mr{fs}/S)_{\mr{fl}} .
\end{equation}

Kato's multiplicative group (or the log multiplicative group) $\Gml$ is the sheaf on $\Set$ defined by $\Gml(U)=\Gamma(U,M^{\mr{gp}}_U)$
for any $U\in\fsS$, where $M_U$ denotes the log structure of $U$ and $M^{\mr{gp}}_U$ denotes the group envelope of $M_U$. The Kummer \'etale sheaf $\Gml$ is also a sheaf on $\Skfl$, see \cite[Cor. 2.22]{niz1} for a proof.

By convention, for any sheaf of abelian groups $F$ on $\Skfl$ and a subgroup sheaf $G$ of $F$ on $\Skfl$, we denote by $(F/G)_{\Set}$ (resp. $(F/G)_{\Sfl}$, resp. $(F/G)_{\Sket}$) the quotient sheaf on $\Set$ (resp. $\Sfl$, resp. $\Sket$), while $F/G$ denotes the quotient sheaf on $\Skfl$. We abbreviate the quotient sheaf $\Gml/\Gm$ on $\Skfl$ as $\Gmlb$.

\section{Introduction}\label{sec1}
The notion of 1-motive is introduced by Deligne in \cite{del1}. A 1-motive over a base scheme $S$ is a two term complex $M=[Y\xrightarrow{u} G]$ of group schemes over $S$ such that
\begin{enumerate}[(1)]
\item $Y$ sits in degree -1, and is \'etale locally isomorphic to a finitely generated free abelian group (we call such a group scheme $Y$ a lattice);
\item $G$ sits in degree 0, and is an extension of an abelian scheme $B$ by a torus $T$ over $S$.
\end{enumerate}
From the definition, one sees that a 1-motive over a field is a mixture of lattice, torus and abelian variety, and can be regarded as a mixed motive of weights 0, -1, and -2. The name 1-motive comes from the fact that they are the mixed motives arising from varieties of dimension less or equal to 1. For any positive integer $n$, one can associate to $M$ a finite flat group scheme 
$$T_{\Z/n\Z}(M):=H^{-1}(M\otimes^L\Z/n\Z),$$
and for a prime number $l$ we define the $l$-adic Tate module of $M$ as $$T_l(M):=\varprojlim_rT_{\Z/l^r\Z}(M).$$

Now let $R$ be a complete discrete valuation ring with fraction field $K$, residue field $k$, and a chosen uniformizer $\pi$, $S=\Spec R$, and we endow $S$ with the log structure associated to $\N\rightarrow R,1\mapsto \pi$. Let $s$ (resp. $\eta$) be the closed (resp. generic) point of $S$. We denote by $i:s\hookrightarrow S$ (resp. $j:\eta\hookrightarrow S$) the closed (resp. open) immersion of $s$ (resp. $\eta$) into $S$. We endow $s$ with the induced log structure from $S$. 

A log 1-motive in \cite[\S 4.6.1]{k-t1} is defined to be a triple $M_{\mr{log}}=(Y,G,u_K)$ with $Y$ and $G$ as in the definition of 1-motive, but $u_K$ only a homomorphism of group schemes defined over $K$. Then we get a 1-motive $M_K:=[Y_K\xrightarrow{u_K}G_K]$ over $K$ from $M_{\mr{log}}$. In \cite[Thm. 19]{b-c-c1}, the authors extend $T_{\Z/n\Z}(M_K)$ to a log finite group object in $(\mr{fin}/S)_r$ (see Definition \ref{defn5.1}) by using Kato's classification theorem for objects in $(\mr{fin}/S)_{\mr{r}}$ for an fs log scheme $S$ with its underlying scheme a noetherian strictly henselian local ring. Note that $Y_K$ and $G_K$ obviously have good reduction. 

In this paper, a log 1-motive is as defined in \cite[Def. 2.2]{k-k-n2} rather than in \cite[\S4.6.1]{k-t1}, which is the more suitable one over a general base. We are going to show that a 1-motive $M_K=[Y_K\xrightarrow{u_K}G_K]$ with both $Y_K$ and $G_K$ having good reduction, can be extended to a unique log 1-motive $M=[Y\xrightarrow{u}G_{\mr{log}}]$ over $S$, see Corollary \ref{cor3.1}. Hence a log 1-motive in the sense of \cite[\S4.6.1]{k-t1} is a log 1-motive in our sense. Taking $T_{\Z/n\Z}(M)$, we get an object of $(\mr{fin}/S)_{\mr{r}}$ with generic fiber $T_{\Z/n\Z}(M_K)$. This gives an alternative proof to \cite[Thm. 19]{b-c-c1}. Moreover, if we replace log 1-motive by k\'et log 1-motive (see Definition \ref{defn2.6}), we can generalize the result to tamely ramified strict 1-motives over $K$, see Theorem \ref{thm1.1} below. Here a 1-motive $M_K=[Y_K\xrightarrow{u_K}G_K]$ is called strict, if $G_K$ has potentially good reduction (see \cite[Def. 4.2.3]{ray2}), and a 1-motive $M_K=[Y_K\xrightarrow{u_K}G_K]$ is called tamely ramified, if $Y_K$ (resp. $G_K$) has good reduction (resp. semi-stable reduction) after a tamely ramified extension of $K$.

\begin{thm}[See also Theorem \ref{thm3.1}]\label{thm1.1}
Let $M_K=[Y_K\xrightarrow{u_K}G_K]$ be a tamely ramified strict 1-motive over $K$. Then $M_K$ extends to a unique k\'et log 1-motive $M^{\mr{log}}$ over $S$.

Moreover the association of $M^{\mr{log}}$ to $M_K$ gives rise to an equivalence
$$K\acute{e}t:\mr{TameSt\textnormal{-}1\textnormal{-}Mot}_K\to \mr{K\acute{e}tLog\textnormal{-}1\textnormal{-}Mot}_S$$
from the category of tamely ramified strict 1-motives over $K$ to the category of k\'et log 1-motives over $S$.
\end{thm}

Let us make some historical review concerning Theorem \ref{thm1.1}. Without doubt, degeneration is an important topic in mathematics. As stated in the beginning of \cite{k-k-n2}, degenerating abelian varieties cannot preserve smoothness, properness, and group structure at the same time, while the theory of log abelian varieties make the impossible possible \footnote{To be more precise, one can make the impossible possible up to certain extent at least at this moment, for example over a complete discrete valuation field, one can only extend the abelian varieties with semi-stable reduction to log abelian varieties over the corresponding discrete valuation ring endowed with the  canonical log structure.} in the world of log geometry. Let us get back to the setup that $R$ is a complete discrete valuation ring with fraction field $K$ and $S:=\Spec R$ is endowed with the canonical log structure. Let $A_K$ be an abelian variety with semi-stable reduction over $K$. Following the ideas from \cite{k-k-n4} and \cite{k-k-n6}, we give a sketch of the construction of the log abelian variety over $S$ extending $A_K$ . We have the Raynaud extension of $A_K$ which is an extension $G$ of some abelian scheme $B$ by some torus $T$ over $S$, see \cite[Exp. IX, Prop. 7.5]{sga7-1}. Let $A_K^\vee$ be the dual abelian variety of $A_K$. Then $A_K^\vee$ has semi-stable reduction too by \cite[Exp. IX, Rmk. 3.5.1]{sga7-1}. Let $G^\vee$ be the Raynaud extension of $A_K^\vee$ which is an extension of the dual $B^\vee$ of $B$ by a torus $T^\vee$ by \cite[Exp. IX, Prop. 7.5]{sga7-1}. Let $Y$ be the character group of $T^\vee$. Then the extension $G^\vee$ of $B^\vee$ by $T^\vee$ corresponds to a homomorphism $v:Y\to (B^\vee)^{\vee}=B$. Let $Y_K:=Y\times_S\Spec K$ and $G_K:=G\times_S\Spec K$. Then there exists a homomorphism $u_K:Y_K\to G_K$ of group schemes over $K$ lifting $v_K:=v\times_S\Spec K$, such that the rigid analytic space of $A_K$ is the quotient of the rigid analytic space of $G_K$ by $Y_K$, see \cite[Exp. IX, \S14.1]{sga7-1}. So we get a 1-motive $M_K=[Y_K\to G_K]$ in which both $Y_K$ and $G_K$ have good reduction (or a log 1-motive in the sense of \cite[\S4.6.1]{k-t1}). The 1-motive $M_K$ extends to a log 1-motive $M^{\mr{log}}=[Y\to G]$ over $S$. Let $\pi$ be a chosen uniformizer of $R$, $S_n:=\Spec R/(\pi)^{n+1}$ endowed with induced log structure from $S$, and $M^{\mr{log}}_n:=M^{\mr{log}}\times_SS_n$. Any polarization $\lambda_K$ of $A_K$ gives rise to a polarization $\lambda_n$ of $M^{\mr{log}}_n$. The polarizations $\lambda_n$ are compatible with each other as $n$ varies. Let $A_n$ be the log abelian variety over $S_n$ associated to the polarized log 1-motive $M^{\mr{log}}_n$. Then we get a polarizable $\pi$-adic formal log abelian variety $\mc{A}$ over $R$, i.e. an object of the category $\mc{P}$ from \cite[\S5.1]{k-k-n6}. Through the equivalence of categories from \cite[Thm. 6.1]{k-k-n6}, we get a log abelian variety $A$ over $S$ whose formal completion is $\mc{A}$. To sum up, we have the following associations
\begin{equation}\label{eq1.1}
A_K\mapsto M_K\mapsto M^{\mr{log}}\mapsto \{M^{\mr{log}}_n\}_n\mapsto \mc{A}\mapsto A.
\end{equation}
Our Theorem \ref{thm1.1} is a generalization of the association $M_K\mapsto M^{\mr{log}}$ to tamely ramified strict 1-motives over $K$. 

In the association chain (\ref{eq1.1}), instead of starting with a semi-stable abelian variety over $K$, we can start with a tamely ramified abelian variety over $K$. Then using Theorem \ref{thm1.1}, we can give a proof to the following theorem (see also Theorem \ref{thm5.2}) which is stated in the preprint \cite[\S 4.3]{kat4} without proof. 

\begin{thm}[Kato]\label{thm1.2}
Let $K$ be a complete discrete valuation field with ring of integers $R$, $p$ a prime number, and $A_K$ a tamely ramified abelian variety over $K$. We endow $S:=\Spec R$ with the canonical log structure. Then the $p$-divisible group $A_K[p^{\infty}]$ of $A_K$ extends to a k\'et log $p$-divisible group, i.e. an object of $(\text{$p$-div}/S)^{\mr{log}}_{\mr{\acute{e}}}$ (see Definition \ref{defn5.2}). It extends to an object of $(\text{$p$-div}/S)^{\mr{log}}_{\mr{d}}$ (see Definition \ref{defn5.2}) if any of the following two conditions is satisfied.
\begin{enumerate}[(1)]
\item $A_K$ has semi-stable reduction.
\item $p$ is invertible in $R$.
\end{enumerate} 
\end{thm}

In the association chain (\ref{eq1.1}), starting with a tamely ramified abelian variety also brings us the question if one can formulate the theory of log abelian varieties in the Kummer \'etale topology instead of the classical \'etale topology, with which one could get a complete association chain as in (\ref{eq1.1}) in the new case. A  second natural question is if one can go further to define log abelian varieties in the even finer topology the Kummer flat topology\footnote{The theory of log abelian varieties in the Kummer flat topology has been expected by Kazuya Kato as Professor Chikara Nakayama informed the author.}. A third question, as suggested by the anonymous referee, is if one can extend a log 1-motive $M_K$ over an fs log point with underlying scheme $\Spec K$ to a log 1-motive over $\Spec R$ endowed with the direct image log structure along $\Spec K\to\Spec R$. We hope to come back to these questions in future.

In Section \ref{sec2}, we define k\'et tori, k\'et lattices, k\'et abelian schemes, k\'et 1-motives, and k\'et log 1-motives, and formulate duality theory for these objects.

Section \ref{sec3} is devoted to the proof of Theorem \ref{thm1.1}. A special case of Theorem \ref{thm1.1} is the following theorem, which gives rise to a concrete non-trivial example of k\'et abelian scheme.

\begin{thm}[See also Theorem \ref{thm3.2}]\label{thm1.3}
Let $K$ be a complete discrete valuation field with ring of integers $R$, and $B_K$ a tamely ramified abelian variety over $K$ which has potentially good reduction. We endow $S:=\Spec R$ with the canonical log structure, then $B_K$ extends to a unique k\'et abelian scheme $B$ over $S$. 
\end{thm}

In Section \ref{sec4}, for a  tamely ramified strict 1-motive $M_K$ as in Theorem \ref{thm1.1}, we associate a logarithmic monodromy pairing to $M$ and compare it with Raynaud's geometric monodromy for $M_K$ (see \cite[\S4.3]{ray2}).

In Section \ref{sec5}, we present a proof to Theorem \ref{thm1.2}.

\section{K\'et log 1-motives}\label{sec2}
\subsection{K\'et log 1-motives}
The following definitions about log 1-motives are taken from \cite[\S 2]{k-k-n2}.
\begin{defn}
Let $S$ be an fs log scheme, $T$ a torus over the underlying scheme of $S$ with character group $X$. The \textbf{log augmentation of $T$}, denoted by $T_{\mr{log}}$, is the sheaf of abelian groups
$$\mc{H}om_{S_{\mr{\acute{e}t}}}(X,\Gml)$$
on $(\mr{fs}/S)_{\mr{\acute{e}t}}$.
Let $G$ be an extension of an abelian scheme $B$ by $T$ over the underlying scheme of $S$. The \textbf{logarithmic augmentation of $G$}, denoted by $G_{\mr{log}}$, is the push-out of $G$ along the inclusion $T\hookrightarrow T_{\mr{log}}$.
\end{defn}

\begin{defn}
A \textbf{log 1-motive} over an fs log scheme $S$ is a two-term complex $M=[Y\xrightarrow{u}G_{\mr{log}}]$ in the category of sheaves of abelian groups on $(\mr{fs}/S)_{\mr{\acute{e}t}}$, with the degree $-1$ term $Y$ an \'etale locally constant sheaf of finitely generated free abelian groups and the degree 0 term $G_{\mr{log}}$ as above. We also call $Y$ the \textbf{lattice part} of $M$. A \textbf{morphism of log 1-motives} is just a homomorphism of complexes.
\end{defn}

By \cite[Prop. 2.1]{zha3}, one can replace $(\mr{fs}/S)_{\mr{\acute{e}t}}$ by $(\mr{fs}/S)_{\mr{k\acute{e}t}}$ in the above definitions. In particular, $T_{\mr{log}}$ and $G_{\mr{log}}$ are sheaves on $(\mr{fs}/S)_{\mr{k\acute{e}t}}$.

Now we define k\'et 1-motives and k\'et log 1-motives, and we work with $(\mr{fs}/S)_{\mr{k\acute{e}t}}$.

\begin{defn}
A \textbf{k\'et (kummer \'etale) lattice} (resp. \textbf{k\'et torus}, resp. \textbf{k\'et abelian scheme}) over an fs log scheme $S$ is a sheaf $F$ of abelian groups on $(\mr{fs}/S)_{\mr{k\acute{e}t}}$ such that the pull-back of $F$ to $S'$ is a lattice (resp. torus, resp. abelian scheme) over $S'$ for some Kummer \'etale cover $S'$ of $S$. Here by a lattice, we mean a group scheme which is \'etale locally representable by a finite rank free abelian group.
\end{defn}

\begin{defn}
Let $S$ be an fs log scheme. A \textbf{k\'et 1-motive} over $S$ is a two-term complex $M=[Y\xrightarrow{u}G]$ in the category of sheaves of abelian groups on $(\mr{fs}/S)_{\mr{k\acute{e}t}}$, such that the degree $-1$ term $Y$ is a k\'et lattice and the degree 0 term $G$ is an extension of a k\'et abelian scheme $B$ by a k\'et torus $T$ on $(\mr{fs}/S)_{\mr{k\acute{e}t}}$.  A \textbf{morphism of k\'et 1-motives} is just a homomorphism of complexes.
\end{defn}

\begin{lem}\label{lem2.1}
Let $S$ be an fs log scheme. Then the associations 
$$T\mapsto\mc{H}om_{S_{\mr{k\acute{e}t}}}(T,\Gm),\quad X\mapsto \mc{H}om_{S_{\mr{k\acute{e}t}}}(X,\Gm)$$
define an equivalence between the category of k\'et tori over $S$ and the category of k\'et locally constant sheaves of finitely generated free abelian groups over $S$, and the equivalence restricts to the classical equivalence between the category of classical tori over $S$ and the category of \'etale locally constant sheaves of finitely generated free abelian groups over $S$. We call the k\'et lattice $\mc{H}om_{S_{\mr{k\acute{e}t}}}(T,\Gm)$ the \textbf{character group} of the k\'et torus $T$.
\end{lem}
\begin{proof}
This follows from the classical equivalence between the category of tori and the category of \'etale locally constant sheaves of finitely generated free abelian groups.
\end{proof}

\begin{defn}\label{defn2.5}
Given a k\'et torus $T$ over $S$, let $X:=\mc{H}om_{S_{\mr{k\acute{e}t}}}(T,\Gm)$ be the character group of $T$. The \textbf{logarithmic augmentation of $T$}, denoted by $T_{\mr{log}}$, is the sheaf of abelian groups
$$\mc{H}om_{S_{\mr{k\acute{e}t}}}(X,\Gml)$$
on $(\mr{fs}/S)_{\mr{k\acute{e}t}}$.
Let $G$ be an extension of a k\'et abelian scheme $B$ by $T$ over $S$. The \textbf{logarithmic augmentation of $G$}, denoted by $G_{\mr{log}}$, is the push-out of $G$ along the inclusion $T\hookrightarrow T_{\mr{log}}$. 
\end{defn}

Note that the quotient $(G_{\mr{log}}/G)_{S_{\mr{k\acute{e}t}}}$ is canonically identified with the quotient $(T_{\mr{log}}/T)_{S_{\mr{k\acute{e}t}}}$, which can be further identified with $\mc{H}om_{S_{\mr{k\acute{e}t}}}(X,(\Gml/\Gm)_{S_{\mr{k\acute{e}t}}})$.

\begin{defn}\label{defn2.6}
A \textbf{k\'et log 1-motive} over an fs log scheme $S$ is a 2-term complex $M=[Y\xrightarrow{u} G_{\mr{log}}]$ of sheaves of abelian groups on $(\mr{fs}/S)_{\mr{k\acute{e}t}}$ such that the degree -1 term $Y$ is a k\'et lattice over $S$ and $G$ is an extension of a k\'et abelian scheme $B$ by a k\'et torus on $(\mr{fs}/S)_{\mr{k\acute{e}t}}$. The composition $$Y\xrightarrow{u} G_{\mr{log}}\rightarrow (G_{\mr{log}}/G)_{S_{\mr{k\acute{e}t}}}=(T_{\mr{log}}/T)_{S_{\mr{k\acute{e}t}}}=\mc{H}om_{S_{\mr{k\acute{e}t}}}(X,(\Gml/\Gm)_{S_{\mr{k\acute{e}t}}})$$
corresponds to a pairing
$$Y\times X\rightarrow (\Gml/\Gm)_{S_{\mr{k\acute{e}t}}}.$$
We call this pairing the \textbf{monodromy pairing} of $M$. A \textbf{morphism of k\'et log 1-motives} is just a homomorphism of complexes.
\end{defn}

\begin{prop}\label{prop2.1}
Let $G$ be an extension of a k\'et abelian scheme $B$ by a k\'et torus $T$ over an fs log scheme $S$. Then $G$ is Kummer \'etale locally an extension of an abelian scheme by a torus for the classical \'etale topology.
\end{prop}
\begin{proof}
Without loss of generality, we may assume that $B$ (resp. $T$) is an abelian scheme (resp. torus) over $S$. Let $\varepsilon:(\mr{fs}/S)_{\mr{k\acute{e}t}}\rightarrow (\mr{fs}/S)_{\mr{\acute{e}t}}$ be the forgetful map between these two sites. The spectral sequence 
$$E_2^{i,j}=\mr{Ext}^{i}_{S_{\mr{\acute{e}t}}}(B,R^j\varepsilon_* T)\Rightarrow \mr{Ext}^{i+j}_{S_{\mr{k\acute{e}t}}}(B,T)$$
gives rise to an exact sequence
$$0\rightarrow \mr{Ext}^{1}_{S_{\mr{\acute{e}t}}}(B,T)\rightarrow \mr{Ext}^{1}_{S_{\mr{k\acute{e}t}}}(B,T)\rightarrow \mr{Hom}_{S_{\mr{\acute{e}t}}}(B,R^1\varepsilon_* T).$$
By this exact sequence, it suffices to show that $\mr{Hom}_{S_{\mr{\acute{e}t}}}(B,R^1\varepsilon_* T)=0$. We may assume that $T=\Gm$, then we are reduced to show that
$\mr{Hom}_{S_{\mr{\acute{e}t}}}(B,R^1\varepsilon_* \Gm)$ vanishes. The proof of the vanishing is an adoption of the proof of \cite[Lem. 6.1.1]{k-k-n2} in our situation. Let $\varphi\in \mr{Hom}_{S_{\mr{\acute{e}t}}}(B,R^1\varepsilon_* \Gm)$, and $U$ any object of $(\mr{fs}/S)$, we show that the map $B(U)\to R^1\varepsilon_* \Gm(U)$ induced by $\varphi$ is trivial. By the same argument as in the proof of \cite[Lem. 6.1.1]{k-k-n2}, we are reduced to the case that $U$ is a log point, i.e. its underlying scheme is the spectrum of a field $K$. Let $p$ be the characteristic of $K$. Then we have 
$$R^1\varepsilon_*\Gm\cong(\Gml/\Gm)_{U_{\mr{\acute{e}t}}}\otimes_\Z(\Q/\Z)'$$
over $U$, where 
$$(\Q/\Z)':=\begin{cases}\varinjlim_{(n,p)=1}\frac{1}{n}\Z/\Z, &\text{ if $p>1$;} \\
\Q/\Z, &\text{ if $p=0$.}
\end{cases}$$
Let $(\mr{st}/U)$ be the full subcategory of $(\mr{fs}/U)$ consisting of all objects which are strict over $U$. The restriction of $(\Gml/\Gm)_{U_{\mr{\acute{e}t}}}\otimes_\Z(\Q/\Z)'$ to $(\mr{st}/U)$ is a locally constant sheaf, and hence is represented by an \'etale group scheme over $U$. Since a homomorphism of group schemes over $U$ from $B\times_SU$ to an \'etale group scheme is trivial, the restriction of $\varphi$ to $(\mr{st}/U)$ is trivial. Hence the homomorphism $B(U)\to R^1\varepsilon_* \Gm(U)$ induced by $\varphi$ is trivial. This finishes the proof.
\end{proof}

\begin{rmk}\label{rmk2.1}
For an abelian scheme $B$ and a torus $T$ over $S$, the same argument as in the proof of Proposition \ref{prop2.1} shows that $\mr{Ext}^{1}_{S_{\mr{fl}}}(B,T)\xrightarrow{\cong} \mr{Ext}^{1}_{S_{\mr{kfl}}}(B,T)$. Furthermore, we have 
$$\mr{Ext}^{1}_{S_{\mr{k\acute{e}t}}}(B,T)\cong \mr{Ext}^{1}_{S_{\mr{\acute{e}t}}}(B,T)\cong\mr{Ext}^{1}_{S_{\mr{fl}}}(B,T)\cong \mr{Ext}^{1}_{S_{\mr{kfl}}}(B,T).$$
\end{rmk}

%\begin{prop}
%Let $S$ be an fs log scheme. Then a k\'et torus (resp. k\'et abelian scheme) $T$ (resp. $B$) is a log algebraic space over $S$ in the second sense as defined in \cite[\S 10.1]{k-k-n4} (\textcolor{red}{?????the definition there contains a small error or not?????? , namely the cover should be a surjection of sheaves for the ket topology. Think more about the definition of log alg sp in the 2nd sense, should it  be the quotient of ket equivalence relations in the etale topology or in the ket topology?}).
%\end{prop}

\subsection{K\'et log 1-motives in the Kummer flat topology}
In this subsection, we assume that the underlying scheme of the base $S$ is locally noetherian. We show that a k\'et log 1-motive can be regarded as a 2-term complex in the category of sheaves for the Kummer flat topology.

\begin{lem}\label{lem2.2}
Let $S$ be an fs log scheme, and let $F$ be a sheaf of abelian groups on $(\mr{fs}/S)_{\mr{k\acute{e}t}}$ such that $F\times_SS'$ is representable by an fs log scheme for some Kummer \'etale cover $S'$ of $S$. Then $F$ is also a sheaf for the Kummer flat topology. In particular, k\'et lattices, k\'et tori, and k\'et abelian schemes over $S$ are sheaves for the Kummer flat topology.
\end{lem}
\begin{proof}
It suffices to prove that, for any $U\in (\mr{fs}/S)$ and any Kummer flat cover $\{U_i\}_{i\in I}$ of $U$, the canonical sequence 
$$0\rightarrow F(U)\rightarrow\prod_{i\in I}F(U_i)\rightarrow \prod_{i,j\in I}F(U_i\times_UU_j)$$
is exact. Let $S'':=S'\times_SS'$, consider the following commutative diagram
$$\xymatrix{
&0\ar[d] &0\ar[d] &0\ar[d]  \\
0\ar[r] &F(U)\ar[r]\ar[d] &\prod_{i\in I}F(U_i)\ar[r]\ar[d] &\prod_{i,j\in I}F(U_i\times_UU_j)\ar[d] \\
0\ar[r] &F(U\times_SS')\ar[r]\ar[d] &\prod_{i\in I}F(U_i\times_SS')\ar[r]\ar[d] &\prod_{i,j\in I}F(U_i\times_UU_j\times_SS')\ar[d] \\
0\ar[r] &F(U\times_SS'')\ar[r] &\prod_{i\in I}F(U_i\times_SS'')\ar[r] &\prod_{i,j\in I}F(U_i\times_UU_j\times_SS'') \\
}$$
with exact columns. Since $F\times_SS'$ is representable by an fs log scheme, so is  $F\times_SS''$. By \cite[Thm. 5.2]{k-k-n4}, both $F\times_SS'$ and $F\times_SS''$ are sheaves for the Kummer flat topology. It follows that the second row and the third row are both exact. Therefore the first row is also exact. This finishes the proof.
\end{proof}

\begin{cor}\label{cor2.1}
Let $S$ be an fs log scheme, and let $G$ be an extension of a k\'et abelian scheme $B$ by a k\'et torus $T$ over $S$. Then the logarithmic augmentation $G_{\mr{log}}$ of $G$ defined in Definition \ref{defn2.5} is a sheaf for the Kummer flat topology.
\end{cor}
\begin{proof}
Since $\Gml$ is a sheaf for the Kummer flat topology by \cite[Thm. 3.2]{kat2} and $X$ is a sheaf for the Kummer flat topology by Lemma \ref{lem2.2}, so is $T_{\mr{log}}=\mc{H}om_{S_{\mr{k\acute{e}t}}}(X,\Gml)$. Let $\delta:(\mr{fs}/S)_{\mr{kfl}}\rightarrow (\mr{fs}/S)_{\mr{k\acute{e}t}}$ be the forgetful map between these two sites. The adjunction $(\delta^*,\delta_*)$ gives rise to the following commutative diagram
$$
\xymatrix{
0\ar[r] &T_{\mr{log}}\ar[r]\ar[d]^{=} &G_{\mr{log}}\ar[r]\ar[d] &B\ar[r]\ar[d]^{=} &0 \\
0\ar[r] &T_{\mr{log}}\ar[r]  &\delta_*\delta^*G_{\mr{log}}\ar[r]  &B \ar[r] &R^1\delta_*T_{\mr{log}}
}$$
with exact rows. The left vertical identification comes from $T_{\mr{log}}$ being a sheaf for the Kummer flat topology. The right vertical identification follows from Lemma \ref{lem2.2}. Since $T_{\mr{log}}$ is Kummer \'etale locally of the form $\Gml^{r}$, we get $R^1\delta_*T_{\mr{log}}=0$ by Kato's logarithmic Hilbert 90, see \cite[\S 5]{kat2}. Therefore the canonical map $G_{\mr{log}}\rightarrow \delta_*\delta^*G_{\mr{log}}$ is an isomorphism, i.e. $G_{\mr{log}}$ is also a sheaf for the Kummer flat topology.
\end{proof}

\subsection{Duality of k\'et abelian schemes}
In this subsection we assume that the underlying scheme of the base $S$ is locally noetherian. We formulate the duality theory for k\'et abelian schemes.

Let $B$ be an abelian scheme over a base scheme $S$, the dual abelian scheme $B^{\vee}$ can be described as $\mc{E}xt^1_{S_{\mr{fl}}}(B,\Gm)$ by the Weil-Barsotti formula. We are going to use this description to define the dual of a given k\'et abelian scheme.

\begin{thm}\label{thm2.1}
Let $S$ be an fs log scheme. For any k\'et abelian scheme $B$ over $S$, let $B^{\vee}:=\mc{E}xt^1_{S_{\mr{kfl}}}(B,\Gm)$. Then we have the following.
\begin{enumerate}[(1)]
\item The sheaf $B^{\vee}$ is a k\'et abelian scheme over $S$.
\item There exists a functorial isomorphism $\iota:B\xrightarrow{\cong} (B^{\vee})^{\vee}$.
\end{enumerate}
 \end{thm}
\begin{proof}
For part (1), we may assume that $B$ is actually an abelian scheme. Let $\varepsilon_{\mr{fl}}:(\mr{fs}/S)_{\mr{kfl}}\rightarrow (\mr{fs}/S)_{\mr{fl}}$ be the forgetful map between these two sites. Let $F_1$ (resp. $F_2$) be a sheaf on $(\mr{fs}/S)_{\mr{fl}}$ (resp. $(\mr{fs}/S)_{\mr{kfl}}$). Then we have 
$$\varepsilon_{\mr{fl}*}\mc{H}om_{S_{\mr{kfl}}}(\varepsilon_{\mr{fl}}^*F_1,F_2)=\mc{H}om_{S_{\mr{fl}}}(F_1,\varepsilon_{\mr{fl}*} F_2).$$
Let $\theta$ be the functor sending $F_2$ to $\varepsilon_{\mr{fl}*}\mc{H}om_{S_{\mr{kfl}}}(\varepsilon_{\mr{fl}}^*F_1,F_2)=\mc{H}om_{S_{\mr{fl}}}(F_1,\varepsilon_{\mr{fl}*} F_2)$. Then we get two Grothendieck spectral sequences
\begin{equation*}
E_2^{p,q}=R^p\varepsilon_{\mr{fl}*}\circ R^q\mc{H}om_{S_{\mr{kfl}}}(\varepsilon_{\mr{fl}}^*F_1,-)\Rightarrow R^{p+q}\theta
\end{equation*} 
and
\begin{equation*}
E_2^{p,q}=R^p\mc{H}om_{S_{\mr{fl}}}(F_1,-)\circ R^q\varepsilon_{\mr{fl}*} \Rightarrow R^{p+q}\theta .
\end{equation*}
These two spectral sequences give rise to two exact sequences
\begin{equation*}
\begin{split}
0&\rightarrow R^1\varepsilon_{\mr{fl}*}\mc{H}om_{S_{\mr{kfl}}}(\varepsilon_{\mr{fl}}^*F_1,F_2)\rightarrow R^1\theta(F_2)\rightarrow \varepsilon_{\mr{fl}*}\mc{E}xt^1_{S_{\mr{kfl}}}(\varepsilon_{\mr{fl}}^*F_1,F_2)  \\
&\rightarrow R^2\varepsilon_{\mr{fl}*}\mc{H}om_{S_{\mr{kfl}}}(\varepsilon_{\mr{fl}}^*F_1,F_2)
\end{split}
\end{equation*}
and
\begin{equation*}
0\rightarrow \mc{E}xt^1_{S_{\mr{fl}}}(F_1,\varepsilon_{\mr{fl}*} F_2)\rightarrow R^1\theta (F_2)\rightarrow \mc{H}om_{S_{\mr{fl}}}(F_1,R^1\varepsilon_{\mr{fl}*} F_2).
\end{equation*}
Let $F_1=B$ and $F_2=\Gm$. Since $\mc{H}om_{S_{\mr{kfl}}}(B,\Gm)=0$ by \cite[Exp. VIII, 3.2.1]{sga7-1}, we get 
$$R^1\theta (\Gm)\cong \varepsilon_{\mr{fl}*}\mc{E}xt_{S_{\mr{kfl}}}^1(B,\Gm),$$
therefore we get an exact sequence
$$0\rightarrow \mc{E}xt^1_{S_{\mr{fl}}}(B,\Gm)\rightarrow \varepsilon_{\mr{fl}*}\mc{E}xt^1_{S_{\mr{kfl}}}(B,\Gm)\rightarrow \mc{H}om_{S_{\mr{fl}}}(B,R^1\varepsilon_{\mr{fl}*} \Gm).$$
We also have 
$$\mc{H}om_{S_{\mr{fl}}}(B,R^1\varepsilon_{\mr{fl}*} \Gm)=\mc{H}om_{S_{\mr{fl}}}(B,(\Gml/\Gm)_{S_{\mr{fl}}}\otimes_{\Z}(\Q/\Z))=0$$
by a similar argument as in the proof of Proposition \ref{prop2.1}. It follows that 
\begin{equation}\label{eq2.1}
\mc{E}xt^1_{S_{\mr{fl}}}(B,\Gm)\xrightarrow{\cong} \varepsilon_{\mr{fl}*}\mc{E}xt^1_{S_{\mr{kfl}}}(B,\Gm).
\end{equation}
By the Weil-Barsotti formula, the sheaf $\mc{E}xt^1_{S_{\mr{fl}}}(B,\Gm)$ is representable by the dual abelian scheme of $B$. This finishes the proof of part (1).

Now we prove part (2). By \cite[Exp. VIII, 3.2.1]{sga7-1}, we have 
$$\mc{H}om_{S_{\mr{kfl}}}(B,\Gm)=\mc{H}om_{S_{\mr{kfl}}}(B^{\vee},\Gm)=0.$$
By \cite[Exp. VIII, 1.1.1, 1.1.4]{sga7-1}, we get
$$\mr{Hom}_{S_{\mr{kfl}}}(B,(B^{\vee})^{\vee})\xleftarrow{\cong}  \mr{Biext}^1_{S_{\mr{kfl}}}(B,B^{\vee};\Gm)\xrightarrow{\cong} \mr{Hom}_{S_{\mr{kfl}}}(B^{\vee},B^{\vee}).$$
Let $\iota:B\rightarrow(B^{\vee})^{\vee}$ be the homomorphism corresponding to $1_{B^{\vee}}$ under the above identification. Note that $\iota$ is the isomorphism giving the duality in the case that $B$ is actually an abelian scheme. Since $B$ is Kummer \'etale locally an abelian scheme, $\iota$ is Kummer \'etale locally an isomorphism. It follows that $\iota$ is also an isomorphism over $S$.
\end{proof}

\begin{defn}
Let $S$ be an fs log scheme, and $B$ a k\'et abelian scheme over $S$. In view of Theorem \ref{thm2.1}, we call $B^{\vee}:=\mc{E}xt^1_{S_{\mr{kfl}}}(B,\Gm)$ the \textbf{dual k\'et abelian scheme} of $B$. The biextension $P\in \mr{Biext}^1_{S_{\mr{kfl}}}(B,B^{\vee};\Gm)$ corresponding to $\iota$ is called the \textbf{Poincar\'e biextension of $(B,B^{\vee})$ by $\Gm$}.
\end{defn}

\begin{rmk}
In view of (\ref{eq2.1}), one can also define the dual of $B$ in the flat topology.
\end{rmk}

\subsection{Duality of k\'et 1-motives}
In this subsection we keep assuming that the underlying scheme of the base $S$ is locally noetherian. We formulate the duality theory for k\'et 1-motives.

First we give an equivalent description of k\'et 1-motives using Poincar\'e biextension, through which we present the duality theory of k\'et 1-motives. We will also describe morphisms of k\'et 1-motives with respect to the new description. The situation is almost the same as in the case of classical 1-motives, see \cite[10.2.12, 10.2.13, 10.2.14]{del1}.

Let $M=[Y\xrightarrow{u}G]$ be a k\'et 1-motive over $S$, where $G$ is an extension $0\rightarrow T\rightarrow G\rightarrow B\rightarrow 0$ of a k\'et abelian scheme $B$ by a k\'et torus $T$ on $(\mr{fs}/S)_{\mr{kfl}}$. For any element $\chi\in X:=\mc{H}om_{S_{\mr{kfl}}}(T,\Gm)$, the push-out of the short exact sequence $0\rightarrow T\rightarrow G\rightarrow B\rightarrow 0$ along $\chi$ gives rise to an element of $B^{\vee}=\mc{E}xt^1_{S_{\mr{kfl}}}(B,\Gm)$, whence a homomorphism 
\begin{equation}\label{eq2.2}
v^{\vee}:X\rightarrow  B^{\vee}.
\end{equation}
Let $v$ be the composition
\begin{equation}\label{eq2.3}
v:Y\xrightarrow{u}G\rightarrow B,
\end{equation}
then $u$ corresponds to a unique section $s:Y\rightarrow v^*G$ of the extension $v^*G\in\mr{Ext}^1_{S_{\mr{kfl}}}(Y,T)$. Consider the following commutative diagram
\begin{equation}\label{eq2.4}
\xymatrix{
&\mr{Biext}^1_{S_{\mr{kfl}}}(B,B^{\vee};\Gm)\ar[d]^{(1_B,v^{\vee})^*} \ar[r]^-{\cong} &\mr{Hom}_{S_{\mr{kfl}}}(B,B)\\
\mr{Ext}^1_{S_{\mr{kfl}}}(B,T)\ar[d]_{v^*}\ar[r]^-{\cong} &\mr{Biext}^1_{S_{\mr{kfl}}}(B,X;\Gm)\ar[d]^{(v,1_X)^*}  \\
\mr{Ext}^1_{S_{\mr{kfl}}}(Y,T)\ar[r]^-{\cong} &\mr{Biext}^1_{S_{\mr{kfl}}}(Y,X;\Gm)
},
\end{equation}
where the horizontal isomorphisms come from 
$$\mc{H}om_{S_{\mr{kfl}}}(B,\Gm)=0=\mc{E}xt^1_{S_{\mr{kfl}}}(X,\Gm)$$
and $\mc{E}xt^1_{S_{\mr{kfl}}}(B^{\vee},\Gm)=B$ with the help of \cite[Exp. VIII, 1.1.4]{sga7-1}. Since $G$ gives rise to $v^{\vee}$, the biextension corresponding to $G$ is $(1_B,v^{\vee})^*P$ and we have the following mapping diagram 
$$\xymatrix{
&P\ar@{<->}[r]\ar@{|->}[d]  &1_B  \\
G\ar@{<->}[r]\ar@{|->}[d]  &(1_B,v^{\vee})^*P\ar@{|->}[d]   \\
v^*G\ar@{<->}[r] &(v,v^{\vee})^*P
}
$$
with respect to the commutative diagram (\ref{eq2.4}). The section $s$ of $v^*G$ corresponds to a section of the biextension $(v,v^{\vee})^*P$ of $(Y,X)$ by $\Gm$, which we still denote by $s$ by abuse of notation. Therefore we get an equivalent description of the k\'et 1-motive $M=[Y\xrightarrow{u}G]$ of the form
\begin{equation}\label{eq2.5}
\xymatrix{
(v\times v^{\vee})^*P\ar[r]\ar[d] &P\ar[d]   \\
Y\times X\ar[r]^{v\times v^{\vee}}\ar@/^/[u]^s  &B\times B^{\vee} 
},
\end{equation}
where $(v\times v^{\vee})^*P$ denotes the pull-back of the Poincar\'e biextension $P$. Note that the section $s$ and the composition $Y\times X\to(v\times v^{\vee})^*P\to P$ determine each other, thus we also denote the composition by $s$ by abuse of notation and use the diagram
\begin{equation}
\xymatrix{
&P\ar[d]   \\
Y\times X\ar[r]^{v\times v^{\vee}}\ar[ru]^s  &B\times B^{\vee} 
}
\end{equation}
to describe $M$. The description (\ref{eq2.5}) is symmetric. If we switch the role of $Y$ and $X$, $v$ and $v^{\vee}$, $B$ and $B^{\vee}$, we get another k\'et 1-motive $M^{\vee}=[X\xrightarrow{u^{\vee}}G^{\vee}]$, where 
\begin{equation}\label{eq2.6}
G^{\vee}\in\mr{Ext}^1_{S_{\mr{kfl}}}(B^{\vee},T^{\vee})
\end{equation}
corresponds to $(v,1_{B^{\vee}})^*P\in\mr{Biext}^1_{S_{\mr{kfl}}}(Y,B^{\vee};\Gm)$ with $T^{\vee}:=\mc{H}om_{S_{\mr{kfl}}}(Y,\Gm)$. The association of $M^{\vee}$ to $M$ is clearly a duality.

\begin{defn}
We call the k\'et 1-motive $M^{\vee}=[X\xrightarrow{u^{\vee}}G^{\vee}]$ the \textbf{dual k\'et 1-motive} of the k\'et 1-motive $M=[Y\xrightarrow{u}G]$.
\end{defn}

We give a description of morphisms of k\'et 1-motives via Poincar\'e biextension.  
Let 
$$(f_{-1},f_0):M=[Y\xrightarrow{u}G]\to[Y'\xrightarrow{u'}G']=M'$$
be a morphism of k\'et 1-motives. 

\begin{lem}\label{lem2.3}
There is no non-trivial homomorphism between a k\'et torus and a k\'et abelian scheme.
\end{lem}
\begin{proof}
This follows from the classical result that there is no non-trivial homomorphism between a torus and an abelian scheme.
\end{proof}

By Lemma \ref{lem2.3}, $f_0$ induces a map $f_{\mr{t}}:T\to T'$  on the torus part and a map $f_{\mr{ab}}:B\to B'$ on the abelian part. Let $f_{-1}^{\vee}:X'\to X$ be the map between the character groups induced by $f_{\mr{t}}$, and let $f_{\mr{ab}}^\vee:B^{'\vee}\to B^\vee$ be the dual of $f_{\mr{ab}}$. Let $v$ (resp. $v^\vee$) be as in (\ref{eq2.3}) (resp. (\ref{eq2.2})), and similarly we define $v'$ and $v^{'\vee}$ for $M'$. Then we get two commutative squares
\begin{equation}\label{eq2.8}
\xymatrix{Y\ar[r]^v\ar[d]_{f_{-1}} &B\ar[d]^{f_{\mr{ab}}} \\  Y'\ar[r]_{v'} &B'} , \quad
\xymatrix{X\ar[r]^{v^\vee} &B^\vee \\   X'\ar[u]^{f_{-1}^{\vee}}\ar[r]_{v^{'\vee}} &B^{'\vee}\ar[u]_{f_{\mr{ab}}^\vee}}.
\end{equation}
Let $P$ (resp. $P'$) be the Poincar\'e biextension of $(B,B^\vee)$ (resp. $(B',B^{'\vee})$) by $\Gm$, and $s:Y\times X\to P$ (resp. $s':Y'\times X'\to P'$) the section corresponding to $M$ (resp. $M'$). Then in the following canonical diagram
\begin{equation}\label{eq2.9}
\xymatrix{
&&&P\ar[ld]  \\
Y\times X\ar[rr]_{v\times v^{\vee}}\ar[rrru]^(.6)s &&B\times B^\vee  \\
&&&(1\times f_{\mr{ab}}^\vee)^*P=(f_{\mr{ab}}\times 1)^*P'=:Q \ar[ld]\ar[uu]\ar[dd]   \\
Y\times X'\ar[rr]_{v\times v^{'\vee}}\ar[uu]_{1\times f_{-1}^{\vee}}\ar[dd]^{f_{-1}\times 1}\ar[rrru]|!{[rr];[rruu]}\hole &&B\times B^{'\vee}\ar[uu]^(.6){1\times f_{\mr{ab}}^\vee}\ar[dd]_(.4){f_{\mr{ab}}\times 1}  \\
&&&P'\ar[ld]  \\
Y'\times X'\ar[rr]_{v'\times v^{'\vee}}\ar[rrru]^(.6){s'}|!{[rr];[rruu]}\hole &&B'\times B^{'\vee}    
},
\end{equation}
the equality $f_0\circ u=u'\circ f_{-1}$ implies that for any $y\in Y$ and any $x'\in X'$ we have
\begin{equation}\label{eq2.10}
s(y,f_{-1}^\vee(x'))=s'(f_{-1}(y),x')
\end{equation}
after identifying the $\Gm$-torsors $P_{v(y),v^{\vee}(f_{-1}^\vee(x'))}$ and $P'_{v'(f_{-1}(y)),v^{'\vee}(x')}$ along the composition 
$$P_{v(y),v^{\vee}(f_{-1}^\vee(x'))}\xleftarrow{\cong}Q_{v(y),v^{'\vee}(x')}\xrightarrow{\cong}P_{v'(f_{-1}(y)),v^{'\vee}(x')}.$$
Conversely, given any two commutative squares as in (\ref{eq2.8}) such that the equality (\ref{eq2.10}) holds with respect to the diagram (\ref{eq2.9}), we get a morphism from $M$ to $M'$ of k\'et 1-motives.

\subsection{Duality of k\'et log 1-motives}
In this subsection we keep assuming that the underlying scheme of the base $S$ is locally noetherian. We formulate the duality theory for k\'et log 1-motives, which is analogous to the case of k\'et 1-motives.

First we give an equivalent description of k\'et log 1-motives using Poincar\'e biextension, through which we present the dual theory of k\'et log 1-motives. We will also describe morphisms of k\'et log 1-motives with respect to the new description. 

Let $M=[Y\xrightarrow{u}G_{\mr{log}}]$ be a k\'et log 1-motive over $S$, where $G$ is an extension $0\rightarrow T\rightarrow G\rightarrow B\rightarrow 0$ of a k\'et abelian scheme $B$ by a k\'et torus $T$ on $(\mr{fs}/S)_{\mr{kfl}}$. For any element $\chi\in X:=\mc{H}om_{S_{\mr{kfl}}}(T,\Gm)$, the push-out of the short exact sequence $0\rightarrow T\rightarrow G\rightarrow B\rightarrow 0$ along $\chi$ gives rise to an element of $B^{\vee}=\mc{E}xt^1_{S_{\mr{kfl}}}(B,\Gm)$, whence a homomorphism 
\begin{equation}\label{eq2.11}
v^{\vee}:X\rightarrow  B^{\vee}.
\end{equation}
Let $v$ be the composition 
\begin{equation}\label{eq2.12}
v:Y\xrightarrow{u}G_{\mr{log}}\rightarrow B,
\end{equation}
then $u$ corresponds to a unique section $s:Y\rightarrow v^*G_{\mr{log}}$ of the extension $v^*G_{\mr{log}}\in\mr{Ext}^1_{S_{\mr{kfl}}}(Y,T_{\mr{log}})$. Consider the following commutative diagram
\begin{equation}\label{eq2.13}
\xymatrix{
&\mr{Biext}^1_{S_{\mr{kfl}}}(B,B^{\vee};\Gml)\ar[d]^{(1_B,v^{\vee})^*} \\
\mr{Ext}^1_{S_{\mr{kfl}}}(B,T_{\mr{log}})\ar[d]_{v^*}\ar[r]^-{\cong} &\mr{Biext}^1_{S_{\mr{kfl}}}(B,X;\Gml)\ar[d]^{(v,1_X)^*}  \\
\mr{Ext}^1_{S_{\mr{kfl}}}(Y,T_{\mr{log}})\ar[r]^-{\cong} &\mr{Biext}^1_{S_{\mr{kfl}}}(Y,X;\Gml)
},
\end{equation}
where the horizontal isomorphisms come from 
$$\mc{E}xt^1_{S_{\mr{kfl}}}(X,\Gml)=0$$ 
with the help of \cite[Exp. VIII, 1.1.4]{sga7-1}. There is an obvious map from the diagram (\ref{eq2.4}) to the diagram (\ref{eq2.13}). Let $P^{\mr{log}}$ be the push-out of $P$ along $\Gm\hookrightarrow\Gml$, and we call it the \textbf{Poincar\'e biextension of $(B,B^\vee)$ by $\Gml$}. Since $G\in \mr{Ext}^1_{S_{\mr{kfl}}}(B,T)$ corresponds to the biextension 
$$(1_B,v^{\vee})^*P\in \mr{Biext}^1_{S_{\mr{kfl}}}(B,X;\Gm),$$
we have that $G_{\mr{log}}\in\mr{Ext}^1_{S_{\mr{kfl}}}(B,T_{\mr{log}})$ corresponds to 
$$(1_B,v^{\vee})^*P^{\mr{log}}\in \mr{Biext}^1_{S_{\mr{kfl}}}(B,X;\Gml).$$
We have the following mapping diagram 
$$\xymatrix{
&P^{\mr{log}}\ar@{|->}[d]  \\
G_{\mr{log}}\ar@{<->}[r]\ar@{|->}[d]  &(1_B,v^{\vee})^*P^{\mr{log}}\ar@{|->}[d]   \\
v^*G_{\mr{log}}\ar@{<->}[r] &(v,v^{\vee})^*P^{\mr{log}}
}
$$
with respect to the commutative diagram (\ref{eq2.13}). The section $s$ of $v^*G_{\mr{log}}$ corresponds to a section of the biextension $(v,v^{\vee})^*P^{\mr{log}}$ of $(Y,X)$ by $\Gml$, which we still denote by $s$ by abuse of notation. Therefore we get an equivalent description of the k\'et log 1-motive $M=[Y\xrightarrow{u}G_{\mr{log}}]$ of the form
\begin{equation}\label{eq2.14}
\xymatrix{
(v\times v^{\vee})^*P^{\mr{log}}\ar[r]\ar[d] &P^{\mr{log}}\ar[d]   \\
Y\times X\ar[r]^{v\times v^{\vee}}\ar@/^/[u]^s  &B\times B^{\vee} 
},
\end{equation}
where $(v\times v^{\vee})^*P^{\mr{log}}$ denotes the pull-back of $P^{\mr{log}}$ along $v\times v^{\vee}$.  Note that the section $s$ and the composition $Y\times X\to(v\times v^{\vee})^*P^{\mr{log}}\to P^{\mr{log}}$ determine each other, and we also denote the composition by $s$ by abuse of notation and use the diagram
\begin{equation}\label{eq2.15}
\xymatrix{
&P^{\mr{log}}\ar[d]   \\
Y\times X\ar[r]^{v\times v^{\vee}}\ar[ru]^s  &B\times B^{\vee} 
}
\end{equation}
to describe $M$. Note that the description (\ref{eq2.14}) is symmetric. If we switch the role of $Y$ and $X$, $v$ and $v^{\vee}$, $B$ and $B^{\vee}$, we get another k\'et log 1-motive $M^{\vee}=[X\xrightarrow{u^{\vee}}G_{\mr{log}}^{\vee}]$, where $G_{\mr{log}}^{\vee}$ is the log augmentation of $G^{\vee}$ (see (\ref{eq2.6})). The association of $M^{\vee}$ to $M$ is clearly a duality.

\begin{defn}
We call the k\'et log 1-motive $M^{\vee}=[X\xrightarrow{u^{\vee}}G_{\mr{log}}^{\vee}]$ the \textbf{dual k\'et log 1-motive} of the k\'et log 1-motive $M=[Y\xrightarrow{u}G_{\mr{log}}]$.
\end{defn}

We give a description of morphisms of k\'et log 1-motives via Poincar\'e biextension.  
Let 
$$(f_{-1},f_0):M=[Y\xrightarrow{u}G_{\mr{log}}]\to[Y'\xrightarrow{u'}G'_{\mr{log}}]=M'$$
be a morphism of k\'et log 1-motives over $S$.
\begin{lem}\label{lem2.4}
The canonical map
$$\mr{Hom}_S(G,G')\to\mr{Hom}_S(G_{\mr{log}},G'_{\mr{log}})$$
is an isomorphism.
\end{lem}
\begin{proof}
Let $\tilde{S}$ be a Kummer \'etale cover of $S$ such that both $G\times_S\tilde{S}$ and $G'\times_S\tilde{S}$ are extensions of a classical abelian scheme by a classical torus over $\tilde{S}$, and let $\tilde{\tilde{S}}:=\tilde{S}\times_S\tilde{S}$. In the following commutative diagram
$$\xymatrix{
0\ar[r] &\mr{Hom}_S(G,G')\ar[r]\ar[d] &\mr{Hom}_{\tilde{S}}(G,G')\ar[r]\ar[d]^\cong &\mr{Hom}_{\tilde{\tilde{S}}}(G,G')\ar[d]^\cong \\
0\ar[r] &\mr{Hom}_S(G_{\mr{log}},G'_{\mr{log}})\ar[r] &\mr{Hom}_{\tilde{S}}(G_{\mr{log}},G'_{\mr{log}})\ar[r] &\mr{Hom}_{\tilde{\tilde{S}}}(G_{\mr{log}},G'_{\mr{log}})
}$$
with exact rows, both the middle vertical map and the right vertical map are isomorphisms by \cite[Prop. 2.5]{k-k-n2}. It follows that the left vertical map is an isomorphism.
\end{proof}

By Lemma \ref{lem2.4}, $f_0$ is induced by a unique homomorphism $f_{\mr{sab}}:G\to G'$ over $S$. Let $f_{\mr{t}}:T\to T'$ be the map induced by $f_{\mr{sab}}$ on the torus part, and $f_{-1}^{\vee}:X'\to X$ the map between the character groups induced by $f_{\mr{t}}$. Let $f_{\mr{ab}}:B\to B'$ be the map induced by $f_{\mr{sab}}$ on the abelian part, and let $f_{\mr{ab}}^\vee:B^{'\vee}\to B^\vee$ be the dual of $f_{\mr{ab}}$. Let $v$ (resp. $v^\vee$) be as in (\ref{eq2.12}) (resp. (\ref{eq2.11})), and similarly we define $v'$ and $v^{'\vee}$ for $M'$. Then we get two commutative squares
\begin{equation}\label{eq2.16}
\xymatrix{Y\ar[r]^v\ar[d]_{f_{-1}} &B\ar[d]^{f_{\mr{ab}}} \\  Y'\ar[r]_{v'} &B'} , \quad
\xymatrix{X\ar[r]^{v^\vee} &B^\vee \\   X'\ar[u]^{f_{-1}^{\vee}}\ar[r]_{v^{'\vee}} &B^{'\vee}\ar[u]_{f_{\mr{ab}}^\vee}}.
\end{equation}
Let $P^{\mr{log}}$ (resp. $P^{'\mr{log}}$) be the Poincar\'e biextension of $(B,B^\vee)$ (resp. $(B',B^{'\vee})$) by $\Gml$, and $s:Y\times X\to P^{\mr{log}}$ (resp. $s':Y'\times X'\to P^{'\mr{log}}$) the section corresponding to $M$ (resp. $M'$). Then in the following canonical diagram
\begin{equation}\label{eq2.17}
\xymatrix{
&&&P^{\mr{log}}\ar[ld]  \\
Y\times X\ar[rr]_{v\times v^{\vee}}\ar[rrru]^(.6)s &&B\times B^\vee  \\
&&&(1\times f_{\mr{ab}}^\vee)^*P^{\mr{log}}=(f_{\mr{ab}}\times 1)^*P^{'\mr{log}}=:Q^{\mr{log}} \ar[ld]\ar[uu]\ar[dd]   \\
Y\times X'\ar[rr]_{v\times v^{'\vee}}\ar[uu]_{1\times f_{-1}^{\vee}}\ar[dd]^{f_{-1}\times 1}\ar[rrru]|!{[rr];[rruu]}\hole &&B\times B^{'\vee}\ar[uu]^(.6){1\times f_{\mr{ab}}^\vee}\ar[dd]_(.4){f_{\mr{ab}}\times 1}  \\
&&&P^{'\mr{log}}\ar[ld]  \\
Y'\times X'\ar[rr]_{v'\times v^{'\vee}}\ar[rrru]^(.6){s'}|!{[rr];[rruu]}\hole &&B'\times B^{'\vee}    
},
\end{equation}
the equality $f_0\circ u=u'\circ f_{-1}$ implies that for any $y\in Y$ and any $x'\in X'$ we have
\begin{equation}\label{eq2.18}
s(y,f_{-1}^\vee(x'))=s'(f_{-1}(y),x')
\end{equation}
after identifying the $\Gml$-torsors $P^{\mr{log}}_{v(y),v^{\vee}(f_{-1}^\vee(x'))}$ and $P^{'\mr{log}}_{v'(f_{-1}(y)),v^{'\vee}(x')}$ along the composition 
$$P^{\mr{log}}_{v(y),v^{\vee}(f_{-1}^\vee(x'))}\xleftarrow{\cong}Q^{\mr{log}}_{v(y),v^{'\vee}(x')}\xrightarrow{\cong}P^{'\mr{log}}_{v'(f_{-1}(y)),v^{'\vee}(x')}.$$
Conversely given any two commutative squares as in (\ref{eq2.16}) such that the equality (\ref{eq2.18}) holds with respect to the diagram (\ref{eq2.17}), we get a morphism from $M$ to $M'$ of k\'et log 1-motives.

\section{Extending tamely ramified strict 1-motives into k\'et log 1-motives}\label{sec3}
From now on, $R$ is a complete discrete valuation ring with fraction field $K$, residue field $k$, and a chosen uniformizer $\pi$, $S=\Spec R$, and we endow $S$ with the log structure associated to $\N\rightarrow R,1\mapsto \pi$. Let $s$ (resp. $\eta$) be the closed (resp. generic) point of $S$, and we denote by $i:s\hookrightarrow S$ (resp. $j:\eta\hookrightarrow S$) the closed (resp. open) immersion of $s$ (resp. $\eta$) into $S$. We endow $s$ with the induced log structure from $S$.

Following \cite[Def. 4.2.3]{ray2}, a 1-motive $M_K=[Y_K\xrightarrow{u_K}G_K]$ over $K$ is called \textbf{strict}, if $G_K$ has potentially good reduction. We call a 1-motive $M_K=[Y_K\xrightarrow{u_K}G_K]$ over $K$ \textbf{tamely ramified}, if there exists a tamely ramified finite field extension $K'$ of $K$ such that $Y_K\times_{\Spec K}\Spec K'$ has good reduction and $G_K\times_{\Spec K}\Spec K'$ has semi-stable reduction (i.e. the connected component of the N\'eron model of $G_K\times_{\Spec K}\Spec K'$ is a semi-abelian scheme). A lattice (resp. torus, resp. abelian variety) over $K$ is called \textbf{tamely ramified}\footnote{For abelian variety, this terminology agrees with the one from \cite[\S 2.1.4]{h-n1}}, if it is so regarded as a 1-motive over $K$. A lattice (resp. torus) over $K$ is called \textbf{unramified}, if it extends to a lattice (resp. torus) over $S$. The main goal is to prove the following theorem.

\begin{thm}\label{thm3.1}
Let $M_K=[Y_K\xrightarrow{u_K}G_K]$ be a tamely ramified strict 1-motive over $K$ with $G_K$ an extension of an abelian variety $B_K$ by a torus $T_K$. Then $M_K$ extends to a unique k\'et log 1-motive $M^{\mr{log}}$ over $S$.

Moreover the association of $M^{\mr{log}}$ to $M$ gives rise to an equivalence
$$K\acute{e}t:\mr{TameSt\textnormal{-}1\textnormal{-}Mot}_K\to \mr{K\acute{e}tLog\textnormal{-}1\textnormal{-}Mot}_S$$
from the category of tamely ramified strict 1-motives over $K$ to the category of k\'et log 1-motives over $S$.
\end{thm}

Before going to the proof of Theorem \ref{thm3.1}, we treat some special cases in the first few subsections.

\subsection{Extending tamely ramified lattices into k\'et lattices}\label{subsec3.1}
For any positive integer $n$, let $S_n$ be the fiber product $S\times_{\Spec\Z[\N]}\Spec\Z[\N]$, where $\Spec\Z[\N]$ is endowed with the log structure associated to the canonical homomorphism $\N\to\Z[\N]$ and the map $\Spec\Z[\N]\to\Spec\Z[\N]$ is induced by the multiplication by $n$ map on the monoid $\N$. The canonical map $S_n\to S$ is a finite Kummer flat cover, and it is even a finite Kummer \'etale cover if $n$ is invertible on $S$. Let $R_n:=R[T]/(T^n-\pi)$. It is easy to see that $R_n$ is also a complete discrete valuation ring and $S_n$ is nothing but $\Spec R_n$ endowed with the log structure associated to $\N\to R_n,1\mapsto T$.

\begin{lem}\label{lem3.1}
Let $H$ be a k\'et lattice (resp. k\'et torus, resp. k\'et abelian scheme) over $S$. Then 
\begin{enumerate}[(1)]
\item there exists a positive integer $n$ such that $n$ is invertible on $S$ and $H\times_SS_n$ is a classical lattice (resp. classical torus, resp. classical abelian scheme) over $S_n$;
\item the generic fiber $H\times_S\Spec K$ is a classical lattice (resp. classical torus, resp. classical abelian variety).
\end{enumerate}
\end{lem}
\begin{proof}
(1) Let $U\to S$ be a Kummer \'etale cover such that $H\times_SU$ is a classical lattice (resp. classical torus, resp. classical abelian scheme) over $U$. By \cite[Prop. 2.15]{niz1}, by passing to a further Kummer \'etale cover of $S$ over $U$, we may assume that $U\to S$ factorizes as $U\to S_n\to S$ for some positive integer $n$ with $U\to S_n$ a classical \'etale cover and $n$ invertible on $S$. It suffices to show that $H\times_SS_n$ is representable by a classical lattice (resp. classical torus, resp. classical abelian scheme) over $S_n$. If $H\times_SU$ is a classical torus over $U$ which is affine over $U$, so is $H\times_SS_n$ by descent with respect to the classical \'etale cover $U\to S_n$. If $H\times_SU$ is a classical lattice over $U$, we are reduced to the torus case by Lemma \ref{lem2.1}. Now we assume that $H\times_SU$ is a classical abelian scheme over $U$. Since $U\to S_n$ is a classical \'etale cover, the morphism $H\times_SU\to H\times_SS_n$ is representable by surjective classical \'etale morphisms. Hence $H\times_SS_n$ is representable by an algebraic space over $S_n$ endowed with the inverse image log structure from $S_n$. By \cite[\href{https://stacks.math.columbia.edu/tag/0422}{Tag 0422}]{stacks-project}, \cite[\href{https://stacks.math.columbia.edu/tag/041W}{Tag 041W}]{stacks-project}, and \cite[\href{https://stacks.math.columbia.edu/tag/041V}{Tag 041V}]{stacks-project}, the morphism $H\times_SS_n\to S_n$ is proper, flat, and of finite presentation. Therefore $H\times_SS_n\to S_n$ is an abelian algebraic space, see \cite[\S 1.1]{del2} for the definition. Since $S_n$ is a trait, $H\times_SS_n\to S_n$ is actually an abelian scheme by \cite[\S 1.1 (b)]{del2}. This finishes the proof of part (1).

(2) Let $H_K:=H\times_S\Spec K$, $K_n:=R_n\otimes_RK$, and $H_{K_n}:=H\times_SS_n\times_{S_n}\Spec K_n$. Then $H_{K_n}=H_K\times_{\Spec K}\Spec K_n$. Since $\Spec K_n$ is a classical \'etale cover of $\Spec K$ and $H_{K_n}$ is a classical lattice (resp. classical torus, resp. classical abelian scheme), so is  $H_K$ by the same argument as in part (1).
\end{proof}

\begin{prop}\label{prop3.1}
Let $Y_K$ be a lattice over $K$, i.e. a group scheme over $K$ which is \'etale locally representable by a finite rank free abelian group. Assume that $Y_K$ is tamely ramified, then $Y_K$ extends to a unique k\'et lattice $Y$ over $S$. 

Furthermore, the association of $Y$ to $Y_K$ gives rise to an equivalence of categories
$$\mr{TameLat}_K\to\mr{K\acute{e}tLat}_S, Y_K\mapsto Y$$
with inverse $Y\mapsto Y\times_S\Spec K$, where $\mr{TameLat}_K$ denotes the category of tamely ramified lattices over $\Spec K$ and $\mr{K\acute{e}tLat}_S$ denotes the category of k\'et 
lattices over $S$. And the equivalence restricts to an equivalence from the category of unramified lattices over $\Spec K$ to the category of classical lattices over $S$.
\end{prop}
\begin{proof}
Let $K'$ be a tamely ramified finite Galois field extension of $K$ such that $Y_K\times_{\Spec K}\Spec K'$ is unramified. If necessary, by enlarging $K'$ by a further unramified extension, we may assume that $Y_K\times_{\Spec K}\Spec K'$ is constant. Let $R'$ be the integral closure of $R$ in $K'$ and $\pi'$ a uniformizer of $R'$. We endow $S':=\Spec R'$ with the log structure associated to $\N\rightarrow R',1\mapsto\pi'$. Then $S'$ is a finite Kummer \'etale Galois cover of $S$ with Galois group $\mr{Gal}(K'/K)$. Therefore $Y_K$ extends to a Kummer \'etale locally constant sheaf $Y$ on $S$.

By Lemma \ref{lem3.1} (1), any k\'et lattice becomes a classical lattice after base change to a Kummer \'etale cover $S_n\to S$ for some positive integer $n$. Therefore it corresponds to a finite rank free abelian group endowed with a continuous $\pi_1^{\mr{log}}(S)$-action, where $\pi_1^{\mr{log}}(S)$ denotes the log \'etale fundamental group, see \cite[\S 4.6]{ill1}. On the other hand, an object of $\mr{TameLat}_K$ corresponds to a finite rank free abelian group endowed with a continuous $\mr{Gal}(\overline{K}/K)$-action such that the wild inertia subgroup $P$ acts trivially, where $\overline{K}$ denotes a separable closure of $K$. By \cite[\S4.7 (b)]{ill1}, we have a canonical isomorphism $\mr{Gal}(\overline{K}/K)/P\xrightarrow{\cong}\pi_1^{\mr{log}}(S)$. It follows that the functor
$$\mr{TameLat}_K\to\mr{K\acute{e}tLat}_S, Y_K\mapsto Y$$
is an equivalence of categories. Clearly, the equivalence restricts to an equivalence from the category of unramified lattices over $\Spec K$ to the category of classical lattices over $S$.
\end{proof}

\begin{ex}\label{ex3.1}
Let $Y_K\in \mr{TameLat}_K$ which is not unramified. Then it extends to a k\'et lattice over $S$ which is not a classical lattice by Proposition \ref{prop3.1}.
\end{ex}

\subsection{Extending tamely ramified tori into k\'et tori}\label{subsec3.2}
\begin{prop}\label{prop3.2}
Let $T_K$ be a torus over $K$. Assume that $T_K$ is tamely ramified, i.e. there exists a tamely ramified finite field extension $K'$ of $K$ such that $T_K\times_KK'$ has good reduction. Then $T_K$ extends to a k\'et torus $T$ over $S$.

Furthermore, the association of $T$ to $T_K$ gives rise to an equivalence of categories
$$\mr{TameTor}_K\to\mr{K\acute{e}tTor}_S, T_K\mapsto T$$
with inverse $T\mapsto T\times_S\Spec K$, where $\mr{TameTor}_K$ denotes the category of tamely ramified tori over $\Spec K$ and $\mr{K\acute{e}tTor}_S$ denotes the category of k\'et 
tori over $S$. And the equivalence restricts to an equivalence from the category of unramified tori over $\Spec K$ to the category of classical tori over $S$.
\end{prop}
\begin{proof}
Since the functor 
$$\mr{TameTor}_K\to\mr{TameLat}_K,T_K\mapsto X_K:=\mc{H}om_{(\Spec K)_{\mr{\acute{e}t}}}(T_K,\Gm)$$
is an equivalence of categories, the result follows from Lemma \ref{lem2.1} and Proposition \ref{prop3.1}.
\end{proof}

\begin{ex}\label{ex3.2}
Let $T_K\in\mr{TameTor}_K$ which is not unramified. Then it extends to a unique k\'et torus over $S$ which is not a classical torus by Proposition \ref{prop3.2}.
\end{ex}

\subsection{Extending tamely ramified abelian varieties with potentially good reduction into k\'et abelian schemes}\label{subsec3.3}
Let $B_K$ be a tamely ramified abelian variety over $K$ which has potentially good reduction, and let $K'$ be a tamely ramified finite Galois field extension of $K$ such that $B_{K'}:=B_K\times_{\Spec K}\Spec K'$ has good reduction. Let $R'$ be the integral closure of $R$ in $K'$. Then $B_{K'}$ extends to an abelian scheme $B'$ over $S':=\Spec R'$. Let $\pi'$ be a uniformizer of $R'$, and we endow $S'$ with the log structure associated to $\N\rightarrow R',1\mapsto\pi'$. Then $S'$ is a finite Galois Kummer \'etale cover of $S$ with Galois group $\Gamma:=\mr{Gal}(K'/K)$. Let $\rho:\Gamma\times S'\rightarrow S'$ be the canonical action of $\Gamma$ on $S'$. Then the morphism $(\rho,\mr{pr}_2):\Gamma\times S'\rightarrow S'\times_SS'$ is an isomorphism. By \cite[\S 1.2, Prop. 8]{b-l-r1}, $B'$ is the N\'eron model of $B_{K'}$. By the universal property of N\'eron model, the $\Gamma$-action on $B_{K'}$ extends to a unique $\Gamma$-action
\begin{equation}\label{eq3.1}
\tilde{\rho}:\Gamma\times B'\rightarrow B'
\end{equation}
on $B'$ which is compatible with the $\Gamma$-action $\rho$ on $S'$ and the group structure of $B'$. We endow $B'$ with the induced log structure from $S'$.

Let $p'$ denote the structure morphism $B'\rightarrow S'$, $\alpha$ denote the morphism $S'\rightarrow S$, and $p:=\alpha\circ p'$. For any $U\in(\mr{fs}/S)$ and any $(a,b)\in (B'\times_SB')(U)$, we have $\alpha(p'(a))=\alpha(p'(b))$. Hence there exists a unique $\gamma\in\Gamma$ such that $p'(a)=\rho(\gamma,p'(b))$. Since $p'(\tilde{\rho}(\gamma,b))=\rho(\gamma,p'(b))=p'(a)$, we get $(a,\tilde{\rho}(\gamma,b))\in (B'\times_{S'}B')(U)$. We define a morphism 
$$\Phi: B'\times_SB'\rightarrow \Gamma\times (B'\times_{S'}B')$$
by sending $(a,b)$ to $(\gamma^{-1},(a,\tilde{\rho}(\gamma,b))$. 

\begin{lem}
The morphism $\Phi$ is an isomorphism with inverse 
$$\Psi:\Gamma\times (B'\times_{S'}B')\rightarrow B'\times_SB',\quad (\gamma,(a,b))\mapsto (a,\tilde{\rho}(\gamma,b))$$
for any $U\in(\mr{fs}/S)$, any $(a,b)\in (B'\times_{S'}B')(U)$, and any $\gamma\in\Gamma$.
\end{lem}
\begin{proof}
Clearly $\Phi$ and $\Psi$ are inverse to each other.
\end{proof}

\begin{lem}\label{lem3.3}
The canonical morphism 
\begin{equation}
(\tilde{\rho},\mr{pr}_2):\Gamma\times B'\rightarrow B'\times_SB'
\end{equation}
is a monomorphism of sheaves on $(\mr{fs}/S)_{\mr{k\acute{e}t}}$.
\end{lem}
\begin{proof}
The composition
$$\Gamma\times B'\xrightarrow{(\tilde{\rho},\mr{pr}_2)} B'\times_SB'\xrightarrow {\iota}B'\times_SB'\xrightarrow {\Phi}\Gamma\times (B'\times_{S'}B')$$
is identified with the morphism $1_{\Gamma}\times\Delta_{B'/S'}$, where $\iota$ denotes the morphism switching the two factors and $\Delta_{B'/S'}$ denotes the diagonal embedding. Therefore the result follows.
\end{proof}

By \cite[\href{https://stacks.math.columbia.edu/tag/0234}{Tag 0234}]{stacks-project}, the action $\tilde{\rho}$ defines a groupoid scheme over $S$, hence by \cite[\href{https://stacks.math.columbia.edu/tag/0232}{Tag 0232}]{stacks-project} the morphism 
$$(\tilde{\rho},\mr{pr}_2):\Gamma\times B'\rightarrow B'\times_SB'$$
is a pre-equivalence relation. Moreover, $(\tilde{\rho},\mr{pr}_2)$ is an equivalence relation by Lemma \ref{lem3.3}. The morphism 
$$(\rho,\mr{pr}_2):\Gamma\times S'\rightarrow S'\times_SS'$$
being an isomorphism is clearly an equivalence relation.

Now we are following \cite[\href{https://stacks.math.columbia.edu/tag/02VE}{Tag 02VE}]{stacks-project} to construct a k\'et abelian scheme over $S$. We remark that 
although the setup there does not agree with ours, the proofs there work verbatim in our case.

Following the approach of \cite[\href{https://stacks.math.columbia.edu/tag/02VG}{Tag 02VG}]{stacks-project}, we take the quotient sheaves for the equivalence relations $(\tilde{\rho},\mr{pr}_2)$ and $(\rho,\mr{pr}_2)$ on the site $(\mr{fs}/S)_{\mr{k\acute{e}t}}$. Since $(\rho,\mr{pr}_2)$ is an isomorphism, the corresponding quotient sheaf is representable by the terminal object $S$. Let $B=K\acute{e}t(B_K)$ be the quotient sheaf for the equivalence relation $(\tilde{\rho},\mr{pr}_2)$. Since the two equivalence relations are compatible with each other, we get a morphism $B\rightarrow S$. Since the equivalence relation $(\tilde{\rho},\mr{pr}_2)$ is compatible with the group structure of $B'$, the quotient sheaf $B'$ carries a structure of sheaf of abelian groups. The verbatim translations of the proof of \cite[\href{https://stacks.math.columbia.edu/tag/045Y}{Tag 045Y}]{stacks-project} and the proof of \cite[\href{https://stacks.math.columbia.edu/tag/07S3}{Tag 07S3}]{stacks-project} show that 
\begin{equation}\label{eq3.3}
\Gamma\times B'\xrightarrow{\cong}B'\times_BB'
\end{equation}
and 
\begin{equation}\label{eq3.4}
B'\xrightarrow{\cong}B\times_SS'
\end{equation}
respectively, hence $B$ is a k\'et abelian scheme over $S$. 

Let $C_K$ be another tamely ramified abelian variety over $K$ which has potentially good reduction, and let $f_K:B_K\to C_K$ be a homomorphism of group schemes over $K$. By enlarging the field $K'$ such that $K'$ remains a finite tamely ramified extension of $K$, we may assume that $C_{K'}:=C_K\times_{\Spec K}\Spec K'$ extends to an abelian scheme $C'$ over $S'$. By the same reason as for $B_K$, we have a natural action of $\Gamma$ on $C'$ which is compatible with the natural action of $\Gamma$ on $S'$ and the quotient for the action gives rise to a k\'et abelian scheme over $S$. Let $f_{K'}$ be the base change of $f_K$ to $K'$. Then it extends to a unique homomorphism $f':B'\to C'$ of group schemes over $S'$ by the universal property of N\'eron model. For any $\gamma\in\Gamma$, the compatibility of the canonical actions of $\Gamma$ on $B_{K'}$ and $C_{K'}$ gives rise to a commutative diagram
$$\xymatrix{
&C_{K'}\ar[r]^\gamma\ar[ldd]|!{[ld];[d]}\hole &C_{K'}\ar[ldd] \\
B_{K'}\ar[r]^(.4)\gamma\ar[ru]^{f_{K'}}\ar[d] &B_{K'}\ar[ru]^{f_{K'}}\ar[d] \\
\Spec K'\ar[r]_\gamma &\Spec K'
}.$$
Hence the following diagram
$$\xymatrix{
&C'\ar[r]^\gamma\ar[ldd]|!{[ld];[d]}\hole &C'\ar[ldd] \\
B'\ar[r]^(.4)\gamma\ar[ru]^{f'}\ar[d] &B'\ar[ru]^{f'}\ar[d] \\
S'\ar[r]_\gamma &S'
}$$
is commutative by the universal property of N\'eron model for any $\gamma\in\Gamma$. It follows that the $\Gamma$-actions on $B'$ and $C'$ are compatible with each other, hence $f'$ induces a homomorphism $f:B\to C$ of k\'et abelian schemes over $S$, where $C=K\acute{e}t(C_K)$ is the k\'et abelian scheme for $C_K$ defined in the same way as $B=K\acute{e}t(B_K)$ for $B_K$.

Therefore we get a functor
$$K\acute{e}t:\mr{TameStAb}_K\rightarrow\mr{K\acute{e}tAb}_S,\quad B_K\mapsto K\acute{e}t(B_K)$$
from the category of tamely ramified abelian varieties with potentially good reduction over $K$ to the category of k\'et abelian schemes over $S$.

\begin{thm}\label{thm3.2}
\begin{enumerate}[(1)]
\item Let $B$ and $C$ be two k\'et abelian schemes over $S$, and let $B_K:=B\times_S\Spec K$ and $C_K:=C\times_S\Spec K$. Then the restriction map 
$$\mr{Hom}_S(B,C)\to\mr{Hom}_{\Spec K}(B_K,C_K)$$
is an isomorphism.
\item The functor
$$K\acute{e}t:\mr{TameStAb}_K\rightarrow\mr{K\acute{e}tAb}_S,\quad B_K\mapsto K\acute{e}t(B_K)$$
is an equivalence of categories, and it restricts to an equivalence from the category of abelian varieties with good reduction over $K$ to the category of abelian schemes over $S$.
\end{enumerate}
\end{thm}
\begin{proof}
(1) By Lemma \ref{lem3.1}, there exists a positive integer $n$ such that both $B_n:=B\times_SS_n$ and $C_n:=C\times_SS_n$ are classical abelian schemes over $S_n$, where $S_n$ is as defined in the beginning of Subsection \ref{subsec3.1}. Let $\tilde{S}_n:=S_n\times_SS_n$, $\tilde{B}_n:=B\times_S\tilde{S}_n$, $\tilde{C}_n:=C\times_S\tilde{S}_n$, $K_n:=S_n\times_S\Spec K$, $\tilde{K}_n:=\tilde{S}_n\times_S\Spec K$, $B_{K_n}:=B_n\times_S\Spec K$, $C_{K_n}:=C_n\times_S\Spec K$, $\tilde{B}_{K_n}:=\tilde{B}_n\times_S\Spec K$, and $\tilde{C}_{K_n}:=\tilde{C}_n\times_S\Spec K$. Since $S_n$ is a Kummer \'etale cover of $S$, we get the following commutative diagram
$$\xymatrix{
0\ar[r] &\mr{Hom}_S(B,C)\ar[r]\ar[d] &\mr{Hom}_{S_n}(B_n,C_n)\ar[r]\ar[d]^{\cong} &\mr{Hom}_{\tilde{S}_n}(\tilde{B}_n,\tilde{C}_n)\ar[d]^{\cong} \\
0\ar[r] &\mr{Hom}_K(B_K,C_K)\ar[r] &\mr{Hom}_{K_n}(B_{K_n},C_{K_n})\ar[r] &\mr{Hom}_{\tilde{K}_n}(\tilde{B}_{K_n},\tilde{C}_{K_n})
}$$
with exact rows. Here we abbreviate $\mr{Hom}_{\Spec ?}(-,-)$ as $\mr{Hom}_{?}(-,-)$ for $?=K, K_n, \tilde{K}_n$, in order to make the diagram not too big. The middle vertical map and the right vertical map are isomorphisms by the universal property of N\'eron model, it follows that the first vertical map is also an isomorphism.

(2) By part (1), the functor $K\acute{e}t$ is fully faithful. Given any $B\in \mr{K\acute{e}tAb}_S$, we have $B_K:=B\times_S\Spec K\in \mr{TameStAb}_K$ by Lemma \ref{lem3.1} (2). By part (1), the identity map on $B_K$ extends to a unique isomorphism $K\acute{e}t(B_K)\xrightarrow{\cong}B$. Hence the functor $K\acute{e}t$ is essentially surjective, and so an equivalence of categories. In fact, it is easy to see that the functor
$$(-)_K:\mr{K\acute{e}tAb}_S\to \mr{TameStAb}_K,B\mapsto B\times_S\Spec K$$
is inverse to the functor $K\acute{e}t$ by the construction of $K\acute{e}t$. If $A_K\in\mr{TameStAb}_K$ has good reduction, i.e. it extends to an abelian scheme $A$ over $S$ (which is unique by the theory of N\'eron models), then there exists a unique isomorphism $A\to K\acute{e}t(A_K)$ extending $\mr{id}_{A_K}$. Therefore $K\acute{e}t(A_K)=A$, and so the functor $K\acute{e}t$ restricts to an equivalence from the category of abelian varieties with good reduction over $K$ to the category of abelian schemes over $S$.
\end{proof}

\begin{ex}\label{ex3.3}
Let $B_K\in\mr{TameStAb}_K$ which has no good reduction. Then $K\acute{e}t(B_K)$ is a k\'et abelian scheme over $S$ which is not a classical abelian scheme by Theorem \ref{thm3.2} (2). 
\end{ex}

It is natural to investigate if the functor $K\acute{e}t$ is compatible with the dualities on both sides.

\begin{prop}
The functor $K\acute{e}t:\mr{TameStAb}_K\rightarrow\mr{K\acute{e}tAb}_S$ is compatible with the dualities, i.e. we have a canonical identification 
$$K\acute{e}t(B_K^{\vee})\cong K\acute{e}t(B_K)^{\vee}.$$
\end{prop}
\begin{proof}
Let $S'$, $\Gamma$, $B'$, and $B$ be as in (\ref{eq3.3}) and (\ref{eq3.4}). Then $B=K\acute{e}t(B_K)$. By (\ref{eq3.4}), we have 
\begin{align*}
B^{\vee}\times_SS'=&\mc{E}xt^1_{S_{\mr{kfl}}}(B,\Gm)\times_SS'=\mc{E}xt^1_{S'_{\mr{kfl}}}(B\times_SS',\Gm)  \\
=&\mc{E}xt^1_{S'_{\mr{kfl}}}(B',\Gm)=B^{'\vee}.
\end{align*}
It follows that $B^{\vee}=\mc{E}xt^1_{S_{\mr{k\acute{e}t}}}(B,\Gm)$ is the quotient sheaf for a descent data with respect to the Galois Kummer \'etale cover $S'/S$. Such a descent data is given by a group action $\tau:\Gamma\times B^{'\vee}\rightarrow B^{'\vee}$. In order to have the identification $K\acute{e}t(B_K^{\vee})\cong K\acute{e}t(B_K)^{\vee}=B^{\vee}$, we are reduced to identify the action $\tau$ with the action $\tilde{\rho}^{\vee}:\Gamma\times B^{'\vee}\rightarrow B^{'\vee}$ for $B_K^{\vee}$ which corresponds to the action (\ref{eq3.1}) for $B_K$. But this is clear, since $\Gamma\times B^{'\vee}=\sqcup_{\gamma\in\Gamma}B^{'\vee}$ and these two actions agree over the generic fiber.
\end{proof}

\subsection{Proof of Theorem \ref{thm3.1}}
In this subsection, we prove Theorem \ref{thm3.1}.

Let $v_K$ be the composition $Y_K\xrightarrow{u_K}G_K\rightarrow B_K$, $X_K$ the character group of the torus $T_K$, and $v_K^{\vee}:X_K\rightarrow B_K^{\vee}$ the homomorphism corresponding to the semi-abelian variety $G_K$. By \cite[2.4.1]{ray2}, the 1-motive $M_K$ is uniquely determined by a commutative diagram of the form
\begin{equation}\label{eq3.5}
\xymatrix{
&P_K\ar[d]  \\
Y_K\times_{\Spec K}X_K\ar[r]^-{v_K\times v_K^{\vee}}\ar[ru]^{s_K} &B_K\times_{\Spec K} B_K^{\vee}
},
\end{equation}
where $s_K$ is a bilinear map. Note that $s_K$ corresponds to a unique section 
\begin{equation}\label{eq3.6}
t_K:Y_K\times_{\Spec K}X_K\rightarrow E_K,
\end{equation}
where $E_K$ denotes the pull-back of the Poincar\'e biextension $P_K$ of $B_K$ and its dual $B_K^{\vee}$ along $v_K\times v_K^{\vee}$. 

Let $K'$ be a finite tamely ramified Galois extension of $K$ such that $B_K\times_{\Spec K}\Spec K'$ extends to an abelian scheme $B'$ over $S':=\Spec R'$, $Y_K\times_{\Spec K}\Spec K'$ extends to a constant group scheme over $S'$, and $T_K\times_{\Spec K}\Spec K'$ extends to a split torus over $S'$, where $R'$ denotes the integral closure of $R$ in $K'$. Let $\pi'$ be a uniformizer of $R'$ such that $\pi'=\pi^{\frac{1}{e}}$ with $e$ the ramification index of the extension $K'/K$, and we endow $S'$ with the log structure associated to $\N\rightarrow R',1\mapsto\pi'$. Then $S'$ is a finite Galois Kummer \'etale cover of $S$ with Galois group $\Gamma:=\mr{Gal}(K'/K)$.

Let $Y$ (resp. $X$) be the k\'et lattice over $S$ extending $Y_K$ (resp. $X_K$) as constructed in Subsection \ref{subsec3.1}. Then $Y$ (resp. $X$) can be regarded as a $\Gamma$-module. Let $T$ be the k\'et torus over $S$ extending $T_K$ as constructed in Subsection \ref{subsec3.2}. Note that $T$ is nothing but $\mc{H}om_{S_{\mr{k\acute{e}t}}}(X,\Gm)$. Let $B$ (resp. $B^{\vee}$) be the k\'et abelian scheme extending $B_K$ (resp. $B_K^{\vee}$) as constructed in Subsection \ref{subsec3.3}, and let $P$ be the Poincar\'e biextension of $(B,B^{\vee})$ by $\Gm$.

\begin{lem}\label{lem3.4}
The homomorphism $v_K$ (resp. $v_K^{\vee}$) extends to a unique homomorphism $v:Y\rightarrow B$ (resp. $v^{\vee}:X\rightarrow B^{\vee}$).
\end{lem}
\begin{proof}
We only treat the case of $v_K$, the other one can be done in the same way. We have $B\times_SS'=B'$ by (\ref{eq3.4}). Therefore 
$$B(S')=B'(S')=B'(\Spec K')=B_K(\Spec K').$$
Since $Y$ is equivalent to a $\Gamma$-module, we get 
$$\mr{Hom}_S(Y,B)=\mr{Hom}_{\Z-\mr{Mod}}(Y,B(S'))^{\Gamma}=\mr{Hom}_{\Z-\mr{Mod}}(Y_K,B_K(\Spec K'))^{\Gamma}.$$
It follows that $v_K$ extends to a unique homomorphism $v:Y\rightarrow B$.
\end{proof}

By Lemma \ref{lem3.4}, we get a map $v\times v^{\vee}:Y\times_SX\rightarrow B\times_S B^{\vee}$. Let $P^{\mr{log}}$ be the push-out of $P$ along the inclusion $\Gm\hookrightarrow\Gml$, we get the following diagram
\begin{equation}\label{eq3.7}
\xymatrix{
&P^{\mr{log}}\ar[d]  \\
Y\times_SX\ar[r]^-{v\times v^{\vee}}\ar@{..>}[ru]^{s_K} &B\times_S B^{\vee}
}
\end{equation}
over $S$. The dotted arrow in (\ref{eq3.7}) means that it is only a map over $\Spec K$. The restriction of (\ref{eq3.7}) to $\Spec K$ is clearly just the diagram (\ref{eq3.5}).

\begin{lem}
The bilinear map $s_K$ from (\ref{eq3.5}) extends uniquely to a bilinear map 
$s^{\mr{log}}:Y\times_SX\rightarrow P^{\mr{log}}$
making the diagram (\ref{eq3.7}) commutative.
\end{lem}
\begin{proof}
Let $E$ be the pull-back of $P$ long the map $v\times v^{\vee}$ on $(\mr{fs}/S)_{\mr{kfl}}$, and let $E^{\mr{log}}$ be the push-out of $E$ along the canonical map $\Gm\hookrightarrow\Gml$ on $(\mr{fs}/S)_{\mr{kfl}}$. Since both $Y$ and $X$ are Kummer \'etale locally representable by a finitely generated free abelian group, we have 
$$\mr{Biext}^1_{S_{\mr{kfl}}}(Y,X;-)=\mr{Ext}^1_{S_{\mr{kfl}}}(Y\otimes_{\Z}^{\mathbb{L}} X,-)=\mr{Ext}^1_{S_{\mr{kfl}}}(Y\otimes_{\Z} X,-)$$
by \cite[Exp. VII, 3.6.5]{sga7-1}. Therefore $E$ (resp. $E^{\mr{log}}$) can be regarded as an element of $\mr{Ext}^1_{S_{\mr{kfl}}}(Y\otimes_{\Z} X,\Gm)$ (resp. $\mr{Ext}^1_{S_{\mr{kfl}}}(Y\otimes_{\Z} X,\Gml)$), and $E^{\mr{log}}$ is still the push-out of $E$ under these identifications. Similarly, $E_K:=E\times_S\Spec K$ can be regarded as an element of $\mr{Ext}^1_{(\Spec K)_{\mr{fl}}}(Y_K\otimes_{\Z} X_K,\Gm)$. Note that both $E$ and $E^{\mr{log}}$ over $S$ restrict to $E_K$ over $K$. The extensions $E$, $E^{\mr{log}}$, and $E_K$, give rise to exact sequences
\begin{equation}\label{eq3.8}
0\rightarrow \Gm(S')\rightarrow E(S')\rightarrow Y\otimes_{\Z} X(S')\rightarrow H^1_{\mr{kfl}}(S',\Gm),
\end{equation}
\begin{equation}\label{eq3.9}
0\rightarrow \Gml(S')\rightarrow E^{\mr{log}}(S')\rightarrow Y\otimes_{\Z} X(S')\rightarrow H^1_{\mr{kfl}}(S',\Gml),
\end{equation}
and
\begin{equation}\label{eq3.10}
0\rightarrow \Gm(K')\rightarrow E_K(K')\rightarrow Y_K\otimes_{\Z} X_K(K')\rightarrow H^1_{\mr{fl}}(\Spec K',\Gm)
\end{equation}
respectively. Clearly we have 
$$H^1_{\mr{fl}}(S',\Gm)=H^1_{\mr{\acute{e}t}}(S',\Gm)=0$$
and
$$H^1_{\mr{fl}}(\Spec K',\Gm)=H^1_{\mr{\acute{e}t}}(\Spec K',\Gm)=0.$$
The short exact sequence $0\rightarrow\Gm\rightarrow\Gml\rightarrow (\Gml/\Gm)_{S_{\mr{fl}}}\rightarrow0$ gives rise to an exact sequence 
$$\rightarrow H^1_{\mr{fl}}(S',\Gm)\rightarrow H^1_{\mr{fl}}(S',\Gml)\rightarrow H^1_{\mr{fl}}(S',(\Gml/\Gm)_{S_{\mr{fl}}})\rightarrow.$$
Since $H^1_{\mr{fl}}(S',(\Gml/\Gm)_{S_{\mr{fl}}})=H^1_{\mr{fl}}(S',i'_*\Z)=H^1_{\mr{fl}}(s',\Z)=H^1_{\mr{\acute{e}t}}(s',\Z)=0$, where $i'$ denotes the inclusion of the closed point $s'$ of $S'$ into itself, we get $H^1_{\mr{fl}}(S',\Gml)=0$. By Kato's logarithmic Hilbert 90, see \cite[Thm. 3.20]{niz1}, we get 
$$H^1_{\mr{kfl}}(S',\Gml)=H^1_{\mr{fl}}(S',\Gml)=0.$$ 
The exact sequences (\ref{eq3.8}), (\ref{eq3.9}), and (\ref{eq3.10}) fit into the following commutative diagram
\begin{equation}\label{eq3.11}
\xymatrix{
0\ar[r] &\Gm(S')\ar[r]\ar[d] &E(S')\ar[r]\ar[d] &Z(S')\ar@{=}[d]\ar[r]^-{\delta} &H^1_{\mr{kfl}}(S',\Gm)\ar[d] \\
0\ar[r] &\Gml(S')\ar[r]\ar[d] &E^{\mr{log}}(S')\ar[r]\ar[d] &Z(S')\ar[r]\ar[d] &0  \\
0\ar[r] &\Gm(\Spec K')\ar[r] &E_K(\Spec K')\ar[r] &Z_K(\Spec K')\ar[r] &0  
}
\end{equation}
with exact rows\footnote{We will see that the map $\delta$ is zero in Subsection \ref{subsec4.2}.}, where $Z$ and $Z_K$ denote $Y\otimes_\Z X$ and $Y_K\otimes_\Z X_K$ respectively. Since $Y$ and $X$ become constant over $S'$, the map 
$$Z(S')\rightarrow Z_K(\Spec K')$$
is an isomorphism. The map $\Gml(S')\rightarrow \Gm(\Spec K')$ is also an isomorphism. Therefore the restriction map 
$$E^{\mr{log}}(S')\rightarrow E_K(\Spec K')=E^{\mr{log}}(\Spec K')$$
is an isomorphism. We regard $E_K$ as an extension of $Y_K\otimes_\Z X_K$ by $\Gm$, then the section $t_K$ (see (\ref{eq3.6})) of $E_K$ induces a section to the surjection $E^{\mr{log}}(S')\rightarrow Y\otimes_\Z X(S')$. This induced section is clearly $\mr{Gal}(S'/S)$-equivariant, therefore gives rise to a section 
\begin{equation}\label{eq3.12}
t^{\mr{log}}:Y\otimes_\Z X\rightarrow E^{\mr{log}}
\end{equation}
to the extension $E^{\mr{log}}$ of $Y\otimes_\Z X$ by $\Gml$. The homomorphism $t^{\mr{log}}$ is automatically a section to the corresponding biextension $E^{\mr{log}}$ of $(Y,X)$ by $\Gml$. Note that $E^{\mr{log}}$ is also the pull-back of $P^{\mr{log}}$ along $v\times v^{\vee}$, and $t^{\mr{log}}$ gives rise to a bilinear map $s^{\mr{log}}:Y\times_{S}X\rightarrow P^{\mr{log}}$ which extends $s_K$. Clearly we have the following commutative diagram
\begin{equation}\label{eq3.13}
\xymatrix{
&P^{\mr{log}}\ar[d]  \\
Y\times_{S}X\ar[r]^-{v\times v^{\vee}}\ar[ru]^{s^{\mr{log}}} &B\times_S B^{\vee}
}.
\end{equation}
This finishes the proof.
\end{proof}

Now we are ready to prove Theorem \ref{thm3.1}.

\begin{proof}[Proof of Theorem \ref{thm3.1}]  
Step 1: We prove that $M_K$ extends to a k\'et log 1-motive over $S$. Recall that $T=\mc{H}om_S(X,\Gm)$, and let $T_{\mr{log}}:=\mc{H}om_S(X,\Gml)$. We have the following two commutative diagrams
$$\xymatrix{
\mr{Ext}^1_{S_{\mr{kfl}}}(B,T)\ar[d]_{v^*}\ar[r]^-{\cong} &\mr{Biext}^1_{S_{\mr{kfl}}}(B,X;\Gm)\ar[d]^{(v,1_X)^*}  \\
\mr{Ext}^1_{S_{\mr{kfl}}}(Y,T)\ar[r]^-{\cong} &\mr{Biext}^1_{S_{\mr{kfl}}}(Y,X;\Gm)
}$$
and
$$\xymatrix{
\mr{Ext}^1_{S_{\mr{kfl}}}(B,T_{\mr{log}})\ar[d]_{v^*}\ar[r]^-{\cong} &\mr{Biext}^1_{S_{\mr{kfl}}}(B,X;\Gml)\ar[d]^{(v,1_X)^*}  \\
\mr{Ext}^1_{S_{\mr{kfl}}}(Y,T_{\mr{log}})\ar[r]^-{\cong} &\mr{Biext}^1_{S_{\mr{kfl}}}(Y,X;\Gml)
},$$
where the horizontal maps being isomorphisms comes from 
$$\mc{E}xt^1_{S_{\mr{kfl}}}(X,\Gm)=\mc{E}xt^1_{S_{\mr{kfl}}}(X,\Gml)=0$$
with the help of \cite[Exp. VIII, 1.1.4]{sga7-1}. Let $G\in \mr{Ext}^1_{S_{\mr{kfl}}}(B,T)$ (resp. $G_{\mr{log}}\in \mr{Ext}^1_{S_{\mr{kfl}}}(B,T_{\mr{log}})$) be the extension corresponding to the biextension $(1_B,v^{\vee})^*P$ (resp. $(1_B,v^{\vee})^*P^{\mr{log}}$). Then the section $s^{\mr{log}}$ of $E^{\mr{log}}$ gives rise to a homomorphism $u^{\mr{log}}:Y\rightarrow G_{\mr{log}}$ fitting into the following commutative diagram
\begin{equation}\label{eq3.14}
\xymatrix{
&&&Y\ar[d]^v\ar[ld]_{u^{\mr{log}}}   \\
0\ar[r] &T_{\mr{log}}\ar[r] &G_{\mr{log}}\ar[r] &B\ar[r] &0 \\
0\ar[r] &T\ar[r]\ar[u] &G\ar[r]\ar[u] &B\ar[r]\ar@{=}[u] &0
}
\end{equation}
of sheaves of abelian groups on $(\mr{fs}/S)_{\mr{kfl}}$. This gives a two-term complex 
$$Y\xrightarrow{u^{\mr{log}}}G_{\mr{log}}.$$
Since both $X$ and $Y$ are representable by a finitely generated free abelian group over $S'$, we have that $G\times_SS'$ is an extension of the abelian scheme $B\times_SS'$ by the torus $T\times_SS'$ on $(\mr{fs}/S')_{\mr{\acute{e}t}}$ by Remark \ref{rmk2.1} and $u^{\mr{log}}\times_SS':Y\times_SS'\rightarrow G_{\mr{log}}\times_SS'$ is a log 1-motive over $S'$. Therefore  $[Y\xrightarrow{u^{\mr{log}}}G_{\mr{log}}]$ is a k\'et log 1-motive over $S$. Clearly the k\'et log 1-motive $[Y\xrightarrow{u^{\mr{log}}}G_{\mr{log}}]$ extends $M_K$, and its construction is functorial. Therefore we get a functor
$$K\acute{e}t:\mr{TameSt\textnormal{-}1\textnormal{-}Mot}_K\to \mr{K\acute{e}tLog\textnormal{-}1\textnormal{-}Mot}_S.$$

Step 2: We show that $\mr{Mor}(M,M')\xrightarrow{\cong}\mr{Mor}(M_K,M'_K)$ for any two objects $M=[Y\xrightarrow{u}G_{\mr{log}}]$ and $M'=[Y'\xrightarrow{u'}G'_{\mr{log}}]$ in $\mr{K\acute{e}tLog\textnormal{-}1\textnormal{-}Mot}_S$, where $M_K:=M\times_S\Spec K$ and $M'_{K}:=M'\times_S\Spec K$. It suffices to show that any morphism $(\mathfrak{f}_{-1},\mathfrak{f}_0):M_K\to M'_K$ extends to a unique morphism from $M$ to $M'$. We rewrite the k\'et log 1-motive $M$ (resp. $M'$) as $(Y\xrightarrow{v} B,X\xrightarrow{v^\vee}B^\vee,Y\times X\xrightarrow{s}P^{\mr{log}})$ (resp. $(Y'\xrightarrow{v'} B',X'\xrightarrow{v^{'\vee}}B^{'\vee},Y'\times X'\xrightarrow{s'}P^{'\mr{log}})$) according to (\ref{eq2.15}). Let $v_K:=v\times_S\Spec K$, similarly for $X_K$, $Y_K$, and so on. By \cite[\S 10.2.12, 10.2.13, 10.2.14]{del1}, the morphism $(\mathfrak{f}_{-1},\mathfrak{f}_0)$ gives rise to a diagram 
\begin{equation}
\xymatrix{
&&&P_K^{\mr{log}}\ar[ld]  \\
Y_K\times X_K\ar[rr]_{v_K\times v_K^{\vee}}\ar[rrru]^(.6){s_K} &&B_K\times B_K^\vee  \\
&&&(1\times \mathfrak{f}_{\mr{ab}}^\vee)^*P_K^{\mr{log}}=(\mathfrak{f}_{\mr{ab}}\times 1)^*P_K^{'\mr{log}}=:Q^{\mr{log}}_K \ar[ld]\ar[uu]\ar[dd]   \\
Y_K\times X_K'\ar[rr]_{v_K\times v_K^{'\vee}}\ar[uu]_{1\times \mathfrak{f}_{-1}^{\vee}}\ar[dd]^{\mathfrak{f}_{-1}\times 1}\ar[rrru]|!{[rr];[rruu]}\hole &&B_K\times B_K^{'\vee}\ar[uu]^(.6){1\times \mathfrak{f}_{\mr{ab}}^\vee}\ar[dd]_(.4){\mathfrak{f}_{\mr{ab}}\times 1}  \\
&&&P_K^{'\mr{log}}\ar[ld]  \\
Y_K'\times X_K'\ar[rr]_{v_K'\times v_K^{'\vee}}\ar[rrru]^(.6){s'_K}|!{[rr];[rruu]}\hole &&B_K'\times B_K^{'\vee}    
},
\end{equation}
such that
\begin{enumerate}[(1)]
\item the two squares are commutative;
\item for any $y\in Y_K$ and any $x'\in X_K'$ we have 
\begin{equation}\label{eq3.16}
s_K(y,\mathfrak{f}_{-1}^\vee(x'))=s_K'(\mathfrak{f}_{-1}(y),x')
\end{equation}
after identifying the $\Gm$-torsors 
$$(P_K^{\mr{log}})_{v_K(y),v_K^{\vee}(\mathfrak{f}_{-1}^\vee(x'))}$$
and 
$$(P_K^{'\mr{log}})_{v_K'(\mathfrak{f}_{-1}(y)),v_K^{'\vee}(x')}$$
along the composition 
$$(P_K^{\mr{log}})_{v_K(y),v_K^{\vee}(\mathfrak{f}_{-1}^\vee(x'))}\xleftarrow{\cong}(Q_K^{\mr{log}})_{v_K(y),v_K^{'\vee}(x')}\xrightarrow{\cong}(P_K^{'\mr{log}})_{v_K'(\mathfrak{f}_{-1}(y)),v_K^{'\vee}(x')}.$$
\end{enumerate}
By Proposition \ref{prop3.1}, $\mathfrak{f}_{-1}$ (resp. $\mathfrak{f}_{-1}^\vee$) extends to a unique homomorphism $f_{-1}:Y\to Y'$ (resp. $f_{-1}^\vee:X'\to X$). By Theorem \ref{thm3.2} (1), $\mathfrak{f}_{\mr{ab}}$ (resp. $\mathfrak{f}_{\mr{ab}}^\vee$) extends to a unique homomorphism $f_{\mr{ab}}:B\to B'$ (resp. $f_{\mr{ab}}^\vee:B^{'\vee}\to B^\vee$). Since $\mathfrak{f}_{\mr{ab}}\circ v_K=v_K'\circ \mathfrak{f}_{-1}$, we have $f_{\mr{ab}}\circ v=v'\circ f_{-1}$ by the same reason as for Lemma \ref{lem3.4}. Similarly we have $f_{\mr{ab}}^\vee\circ v^{'\vee}=v^\vee\circ f_{-1}^\vee$. In order that $((f_{-1},f_{\mr{ab}}),(f_{-1}^\vee,f_{\mr{ab}}^\vee))$ gives a morphism $(f_{-1},f_0)$ from $M$ to $M'$, we are left with checking the condition (\ref{eq2.18}). But this follows from the equality (\ref{eq3.16}) and $\Gamma(\Spec\tilde{R},\Gml)=\Gamma(\Spec \tilde{K},\Gm)$ for any finite field extension $\tilde{K}$ of $K$, where $\tilde{R}$ denotes the normalization of $R$ in $\tilde{K}$ and $\Spec\tilde{R}$ is endowed with the canonical log structure. The uniqueness of $(f_{-1},f_0)$ follows from those of $f_{-1}$, $f_{\mr{ab}}$, $f_{-1}^\vee$, and $f_{\mr{ab}}^\vee$.

Step 3: By Step 1 and Step 2, the functor K\'et is fully faithful. For any $M\in \mr{K\acute{e}tLog\textnormal{-}1\textnormal{-}Mot}_S$, the base change $M_K$ of $M$ to $\Spec K$ lies in $\mr{TameSt\textnormal{-}1\textnormal{-}Mot}_K$. By Step 2, the identity morphism of $M_K$ extends to a unique isomorphism from $M$ to $\mr{K\acute{e}t}(M_K)$. Therefore K\'et is essentially surjective, hence an equivalence of categories. This finishes the proof of Theorem \ref{thm3.1}.
\end{proof}

\begin{ex}
Let $M_K=[Y_K\xrightarrow{u_K}G_K]\in\mr{TameSt\textnormal{-}1\textnormal{-}Mot}_K$, and $G_K$ an extension of an abelian variety $B_K$ by a torus $T_K$ over $K$. Assume that either $Y_K$ is not unramified, or $T_K$ is not unramified, or $B_K$ does not have good reduction. Then the k\'et log 1-motive $M^{\mr{log}}:=K\acute{e}t(M_K)$ is not a log 1-motive by Example \ref{ex3.1}, Example \ref{ex3.2}, and Example \ref{ex3.3}.
\end{ex}

\begin{cor}\label{cor3.1}
Let the notation and the assumptions be as in Theorem \ref{thm3.1}. We further assume that both $Y_K$ and $G_K$ have good reduction. Then the k\'et log 1-motive $M^{\mr{log}}:=K\acute{e}t(M_K)=[Y\xrightarrow{u^{\mr{log}}}G_{\mr{log}}]$ associated to $M_K$ is a log 1-motive.
\end{cor}
\begin{proof}
Since both $Y_K$ and $G_K$ have good reduction, both $X$ and $Y$ are \'etale locally representable by a finitely generated free abelian group over $S$. Therefore $G$ is an extension of the abelian scheme $B$ by the torus $T$ on $(\mr{fs}/S)_{\mr{kfl}}$. By Remark \ref{rmk2.1}, $G$ comes from an extension on $(\mr{fs}/S)_{\mr{\acute{e}t}}$. It follows that $[Y\xrightarrow{u^{\mr{log}}}G_{\mr{log}}]$ is a log 1-motive over $S$.
\end{proof}

\begin{rmk}
Corollary \ref{cor3.1} shows that a log 1-motive in the sense of \cite[4.6.1]{k-t1} extends uniquely to a log 1-motive in our sense (i.e. in the sense of \cite[Defn. 2.2]{k-k-n2}). This must have been known to Kazuya Kato who is one of the authors of both \cite{k-t1} and \cite{k-k-n2}.
\end{rmk}

\begin{rmk}
As we have proposed in Section \ref{sec1}, as a generalization of Theorem \ref{thm3.1}, it is natural to ask if every strict (not just strict tamely ramified) 1-motive over $K$ extends to a unique kfl log 1-motive over $S$, where a kfl log 1-motive can be defined as in Definition \ref{defn2.6} by using the Kummer flat topology instead of the Kummer \'etale topology. In the proofs of Proposition \ref{prop3.1}, Proposition \ref{prop3.2}, and Theorem \ref{thm3.2}, we have used explicitly the theory of logarithmic fundamental group which is only available in the Kummer \'etale topology. Therefore our method in this paper does not work for  the Kummer flat case.
\end{rmk}

\section{Monodromy}\label{sec4}
In this section, we construct a pairing for a tamely ramified strict 1-motive $M_K$ over a complete discrete valuation field via the k\'et log 1-motive $M^{\mr{log}}$ associated to $M_K$. We compare it with the geometric monodromy pairing from \cite[4.3]{ray2}.

\subsection{Logarithmic monodromy pairing}
We adopt the notation from last section. Consider the following push-out diagram
\begin{equation}\label{eq4.1}
\xymatrix{
0\ar[r] &\Gm\ar[r]\ar@{^(->}[d] &E\ar[r]\ar@{^(->}[d] &Y\otimes_{\Z}X\ar[r]\ar@{=}[d] &0  \\
0\ar[r] &\Gml\ar[r] &E^{\mr{log}}\ar[r] &Y\otimes_{\Z}X\ar[r]\ar@/^1pc/[l]^{t^{\mr{log}}} &0
},
\end{equation}
where $t^{\mr{log}}$ is the section (\ref{eq3.12}). Then the section $t^{\mr{log}}$  induces a linear map 
$$Y\otimes_{\Z}X\xrightarrow{t^{\mr{log}}}E^{\mr{log}} \rightarrow E^{\mr{log}}/E\cong (\Gml/\Gm)_{S_{\mr{kfl}}},$$
which corresponds to a bilinear map
\begin{equation}\label{eq4.2}
\langle-,-\rangle:Y\times X\rightarrow (\Gml/\Gm)_{S_{\mr{kfl}}}.
\end{equation}
This pairing is nothing but the monodromy pairing (\ref{defn2.6}) for the k\'et log 1-motive $M^{\mr{log}}$.

\begin{defn}
We call the pairing (\ref{eq4.2}) the \textbf{logarithmic monodromy pairing} of the tamely ramified strict 1-motive $M_K$.
\end{defn}

\begin{prop}\label{prop4.1}
Let the assumption and the notation be as in Theorem \ref{thm3.1} and its proof. The monodromy pairing (\ref{eq4.2}) vanishes if and only if the section $t^{\mr{log}}$ is induced from a section $t:Y\otimes_{\Z}X\rightarrow E$ of $E$.

When such a section $t$ exists, it corresponds to a section $s:Y\times_SX\rightarrow P$ which further corresponds to a map $u:Y\rightarrow G$. The maps $s$ and $u$ extend the diagrams (\ref{eq3.13}) and (\ref{eq3.14}) to the commutative diagrams
\begin{equation}
\xymatrix{
&P\ar[r]\ar[d] &P^{\mr{log}}\ar[d]  \\
Y\times_{S}X\ar[r]_-{v\times v^{\vee}}\ar[ru]^s\ar@{-->}[rru]^(.35){s^{\mr{log}}} &B\times_S B^{\vee}\ar@{=}[r] &B\times_S B^{\vee}
}.
\end{equation}
and
\begin{equation}
\xymatrix{
&&&Y\ar[d]^v\ar[ld]_{u^{\mr{log}}}\ar@{-->}[ldd]_(.35)u   \\
0\ar[r] &T_{\mr{log}}\ar[r] &G_{\mr{log}}\ar[r] &B\ar[r] &0 \\
0\ar[r] &T\ar[r]\ar[u] &G\ar[r]\ar[u] &B\ar[r]\ar@{=}[u] &0
}
\end{equation}
respectively. Therefore the given 1-motive $M_K$ extends to a unique k\'et 1-motive $M=[Y\xrightarrow{u}G]$ such that the k\'et log 1-motive $M^{\mr{log}}$ associated to $M_K$ is induced from $M$.
\end{prop}
\begin{proof}
By the construction the monodromy pairing, its vanishing is clearly equivalent to $t^{\mr{log}}$ being induced from a section $t:Y\otimes_{\Z}X\rightarrow E$ of $E$. The proof of the rest is similar to the proof of Theorem \ref{thm3.1}.
\end{proof}

\begin{prop}\label{prop4.2}
Let $M_K$ be a tamely ramified strict 1-motive over $K$, and $M^{\mr{log}}=[Y\xrightarrow{u^{\mr{log}}}G_{\mr{log}}]$ the k\'et log 1-motive associated to $M_K$. Assume that the logarithmic monodromy pairing of $M_K$ is induced by a pairing $\mu_{\pi}:Y\times X\rightarrow\pi^{\Z}$. Let 
$$u_{2,\pi}^{\mr{log}}:Y\rightarrow T_{\mr{log}}=\mc{H}om_{S_{\mr{kfl}}}(X,\Gml)\subset G_{\mr{log}}$$
be the map induced by $\mu_{\pi}$, and $u_{1,\pi}^{\mr{log}}:=u^{\mr{log}}-u_{2,\pi}^{\mr{log}}$. Then $u_{1,\pi}^{\mr{log}}$ factors as 
$$Y\xrightarrow{u_{1,\pi}}G\hookrightarrow G_{\mr{log}},$$
i.e. the k\'et log 1-motive $[Y\xrightarrow{u_{1,\pi}^{\mr{log}}}G_{\mr{log}}]$ is induced from the k\'et 1-motive $[Y\xrightarrow{u_{1,\pi}}G]$.
\end{prop}
\begin{proof}
It suffices to prove that $u_{1,\pi}^{\mr{log}}$ factors through $G\hookrightarrow G_{\mr{log}}$, the rest is clear. The monodromy pairing of the k\'et log 1-motive $[Y\xrightarrow{u_{1,\pi}^{\mr{log}}}G_{\mr{log}}]$ is the difference of the monodromy pairings of $[Y\xrightarrow{u^{\mr{log}}}G_{\mr{log}}]$ and $[Y\xrightarrow{u_{2,\pi}^{\mr{log}}}G_{\mr{log}}]$. Since the two monodromy pairings agree, we have that the monodromy pairing of $[Y\xrightarrow{u_{1,\pi}^{\mr{log}}}G_{\mr{log}}]$ vanishes. By Proposition \ref{prop4.1}, we are done.
\end{proof}

\begin{ex}
Let $M_K=[Y_K\xrightarrow{u_K} G_K]$ be a tamely ramified strict 1-motive over $K$. Assume that both $Y_K$ and $G_K$ have good reduction. Then both $Y$ and $X$ are \'etale locally constant. Therefore the monodromy pairing 
$$\langle-,-\rangle:Y\times X\rightarrow (\Gml/\Gm)_{S_{\mr{kfl}}}$$ factors through the canonical homomorphism 
$$\pi^{\Z}\cong M^{\mr{gp}}_S/\mc{O}_S^{\times}\rightarrow (\Gml/\Gm)_{S_{\mr{kfl}}}.$$
In other words, the monodromy pairing of $M_K$ satisfies the assumption of Proposition \ref{prop4.2} in this case.
\end{ex}

The construction of the decomposition $u^{\mr{log}}=u_{1,\pi}^{\mr{log}}+u_{2,\pi}^{\mr{log}}$ involves the chosen uniformizer $\pi$. Next we look for a decomposition $u^{\mr{log}}=u_{1}^{\mr{log}}+u_{2}^{\mr{log}}$ independent of the choice of a uniformizer, such that
\begin{equation}\label{eq4.5}
\text{$u_{1}^{\mr{log}}$ is induced by some map $u_{1}:Y\rightarrow G$ and $u_{2}^{\mr{log}}$ factors through $T_{\mr{log}}\hookrightarrow G_{\mr{log}}$.}
\end{equation}

\begin{prop}
Let $M_K$ be a tamely ramified strict 1-motive over $K$, and $M^{\mr{log}}=[Y\xrightarrow{u^{\mr{log}}}G_{\mr{log}}]$ the k\'et log 1-motive associated to $M_K$. The decompositions $u^{\mr{log}}=u_{1}^{\mr{log}}+u_{2}^{\mr{log}}$ satisfying the condition (\ref{eq4.5}) correspond canonically to the trivializations $t:Y\otimes_{\Z} X\rightarrow E$ of the extension $E$ from (\ref{eq4.1}). Then the homomorphism $u_2^{\mr{log}}$ corresponds to the difference homomorphism $t^{\mr{log}}-t$, where $t^{\mr{log}}$ is as in (\ref{eq4.1}).
\end{prop}
\begin{proof}
Given a decomposition $u^{\mr{log}}=u_{1}^{\mr{log}}+u_{2}^{\mr{log}}$ satisfying the condition (\ref{eq4.5}), the map $u_1$ associated to $u_{1}^{\mr{log}}$ gives rise to a section $t:Y\otimes_{\Z} X\rightarrow E$ of $E$. 

Conversely, given a section $t:Y\otimes_{\Z} X\rightarrow E$ of $E$, the decomposition 
$$t^{\mr{log}}=t+(t^{\mr{log}}-t):=t_1+t_2$$
gives rise to a decomposition $u^{\mr{log}}=u_{1}^{\mr{log}}+u_{2}^{\mr{log}}$ with $u_i^{\mr{log}}$ induced by $t_i$. It is clear that $u_1^{\mr{log}}$ factors through $G\hookrightarrow G_{\mr{log}}$. By an easy calculation $t^{\mr{log}}-t$ factors through $\Gml\hookrightarrow E^{\mr{log}}$, therefore $u_2^{\mr{log}}$ factors as $Y\rightarrow T_{\mr{log}}\rightarrow G_{\mr{log}}$. Hence the decomposition $u^{\mr{log}}=u_{1}^{\mr{log}}+u_{2}^{\mr{log}}$ satisfies the condition (\ref{eq4.5}).
\end{proof}

As before, let $Z:=Y\otimes_{\Z} X$. We abbreviate $(\Gml/\Gm)_{S_{\mr{kfl}}}$ as $\Gmlb$. Applying the functor $\mr{Hom}_{S_{\mr{kfl}}}(Z,-)$ to the short exact sequence 
$$0\rightarrow\Gm\rightarrow\Gml\rightarrow\Gmlb\rightarrow0,$$
we get an exact sequence
$$\mr{Hom}_{S_{\mr{kfl}}}(Z,\Gml)\xrightarrow{\alpha}\mr{Hom}_{S_{\mr{kfl}}}(Z,\Gmlb)\rightarrow\mr{Ext}^1_{S_{\mr{kfl}}}(Z,\Gm)\rightarrow\mr{Ext}^1_{S_{\mr{kfl}}}(Z,\Gml).$$
Let $\mu^{\mr{log}}\in\mr{Hom}_{S_{\mr{kfl}}}(Z,\Gmlb)$ be the element corresponding to  the logarithmic monodromy pairing $\langle-,-\rangle$ of $M_K$. Then the element $E$ of $\mr{Ext}^1_{S_{\mr{kfl}}}(Z,\Gm)$ is the image of $\mu^{\mr{log}}$ along the map $\mr{Hom}_{S_{\mr{kfl}}}(Z,\Gmlb)\rightarrow\mr{Ext}^1_{S_{\mr{kfl}}}(Z,\Gm)$. If $E$ is trivial, then the subset 
$$\alpha^{-1}(\mu^{\mr{log}})\subset \mr{Hom}_{S_{\mr{kfl}}}(Z,\Gml)=\mr{Hom}_{S_{\mr{kfl}}}(Y,T_{\mr{log}})$$
is not empty, and its elements correspond to the choices of $u_2^{\mr{log}}$.

\subsection{Comparison with Raynaud's geometric monodromy}\label{subsec4.2}
Since $B$ and $B^{\vee}$ become abelian schemes after base change to $S'$, $P\times_SS'$ is the Poincar\'e biextension of the abelian schemes $B\times_SS'$ and $B^{\vee}\times_SS'$, in particular 
$$P\times_SS'\in\mr{Biext}^1_{S'_{\mr{fl}}}(B\times_SS',B^{\vee}\times_SS';\Gm).$$
It follows that the extension $E\times_SS'$ lies in the subgroup $\mr{Ext}^1_{S'_{\mr{fl}}}((Y\otimes_{\Z} X)\times_SS',\Gm)$ of the group $\mr{Ext}^1_{S'_{\mr{kfl}}}((Y\otimes_{\Z} X)\times_SS',\Gm)$. Therefore the image of the map $\delta$ from (\ref{eq3.11}) lands in the subgroup $H^1_{\mr{fl}}(S',\Gm)$ of $H^1_{\mr{kfl}}(S',\Gm)$. Since $H^1_{\mr{fl}}(S',\Gm)=0$, the diagram (\ref{eq3.11}) gives rise to the following commutative diagram
\begin{equation}\label{eq4.6}
\xymatrix{
0\ar[r] &\Gm(S')\ar[r]\ar[d] &E(S')\ar[r]\ar[d] &Y\otimes_{\Z} X(S')\ar@{=}[d]\ar[r] &0 \\
0\ar[r] &\Gml(S')\ar[r]\ar[d]^{\cong} &E^{\mr{log}}(S')\ar[r]\ar[d]^{\cong} &Y\otimes_{\Z} X(S')\ar[r]\ar[d]^{\cong}\ar@/^1pc/[l]^{t^{\mr{log}}} &0  \\
0\ar[r] &\Gm(\Spec K')\ar[r] &E_K(\Spec K')\ar[r] &Y_K\otimes_{\Z} X_K(\Spec K')\ar[r]\ar@/^1pc/[l]^{t_K} &0  
}
\end{equation}
with exact rows. Then the pairing (\ref{eq4.2}) induces a pairing 
$$Y(S')\times X(S')\rightarrow \Gamma(S',(\Gml/\Gm)_{S_{\mr{kfl}}}),$$
which actually factorizes through a pairing
\begin{equation}\label{eq4.7}
\langle-,-\rangle:Y(S')\times X(S')\rightarrow \Gml(S')/\Gm(S')
\end{equation}
by the diagram (\ref{eq4.6}). The pairing (\ref{eq4.7}) is clearly $\mr{Gal}(S'/S)$-equivariant, and also determines the monodromy pairing (\ref{eq4.2}).
%Since $\Gm(S')=R^{'\times}$ and $\Gml(S')=R^{'\times}\times\pi^{'\Z}$, we get a $\mr{Gal}(S'/S)$-equivariant pairing 
%\begin{equation}
%\langle-,-\rangle:Y(S')\times X(S')\rightarrow \pi^{'\Z}.
%\end{equation}
We have the canonical identification $\Gml(S')/\Gm(S')=(K')^\times/(R')^\times$. Let $\overline{K}$ be a separable closure of $K$ containing $K'$. Then the valuation of $K$ extends to a unique valuation of $\overline{K}$ taking value in $\Q$. Applying the valuation of $\overline{K}$ on $(K')^\times/(R')^\times$, we get a monomorphism $\Gml(S')/\Gm(S')=(K')^\times/(R')^\times\hookrightarrow \Q$. Then the pairing (\ref{eq4.7}) induces a pairing
\begin{equation}\label{eq4.8}
\langle-,-\rangle:Y(S')\times X(S')\rightarrow \Q
\end{equation}
which is equivariant with respect to the action of $\mr{Gal}(S'/S)=\mr{Gal}(K'/K)$.

\begin{prop}
The pairing (\ref{eq4.8}) coincides with the geometric monodromy pairing $\mu:Y_K(\Spec K')\times X_K(\Spec K')\rightarrow\Q$ from \cite[4.3]{ray2}.
\end{prop}
\begin{proof}
The map $t_K$ in the diagram (\ref{eq4.6}) induces a homomorphism 
\begin{align*}
Y_K\otimes_{\Z} X_K(\Spec K')\rightarrow E_K(\Spec K')/E(S')&\cong \Gm(\Spec K')/\Gm(S')  \\
&=(K')^\times/(R')^\times
\end{align*}
which gives rise to exactly the monodromy pairing from \cite[4.3]{ray2} after applying the unique valuation of $\overline{K}$ which extends the valuation on $K$. Since the second row and the third row in the diagram (\ref{eq4.6}) are isomorphic, we are done. 
\end{proof}

If Raynaud's geometric monodromy pairing $\mu$ factors through $\Z\hookrightarrow\Q$, \cite[Prop. 4.5.1]{ray2} gives a decomposition $u_K=u^1_{K,\pi}+u^2_{K,\pi}$ such that
\begin{equation}\label{eq4.9}
\begin{split}
&\text{the $K$-1-motive $M^1_{K,\pi}=[Y_K\xrightarrow{u^1_{K,\pi}}G_K]$ has potentially good reduction;}  \\
&\text{and $u^2_{K,\pi}$ factors through the torus part $T_K$ of $G_K$.}
\end{split}
\end{equation}
Moreover such a decomposition is made independent of the choice of the uniformizer $\pi$ in \cite[Prop. 4.5.3]{ray2}, namely a decomposition $u_K=u^1_K+u^2_K$ satisfying the condition analogous to (\ref{eq4.9}), corresponds to a trivialization $\tau:Z_K=Y_K\otimes_{\Z} X_K\rightarrow \mc{E}_{\mr{rig}}$ of the extension $\mc{E}_{\mr{rig}}$ of $Z_K$ by $U_{\mr{rig}}$ defined in \cite[Rmk. 4.5.2 (iii)]{ray2}. 

\begin{lem}\label{lem4.1}
Raynaud's monodromy factors through $\Z\hookrightarrow\Q$ if and only if the assumption of Proposition \ref{prop4.2} holds, i.e. the logarithmic monodromy pairing (\ref{eq4.2}) of $M_K$ is induced by a pairing $\mu_{\pi}:Y\times X\to\pi^\Z$.
\end{lem}
\begin{proof}
If the logarithmic monodromy is induced by $\mu_{\pi}$, then the image of the composition
\begin{equation}\label{eq4.10}
Y\otimes_{\Z}X(S')\xrightarrow{t^{\mr{log}}}E^{\mr{log}}(S')/E(S')\cong\Gml(S')/\Gm(S')
\end{equation}
lands in $\pi^\Z\Gm(S')/\Gm(S')$. Since $t^{\mr{log}}$ restricts to $t_K$, the image of the composition 
\begin{align*}
Y_K\otimes_{\Z} X_K(\Spec K')\rightarrow E_K(\Spec K')/E(S')&\cong \Gm(\Spec K')/\Gm(S')  \\
&=(K')^\times/(R')^\times
\end{align*}
lands in $\pi^\Z(R')^\times/(R')^\times$. Therefore Raynaud's monodromy factors through $\Z\hookrightarrow\Q$.

Conversely, if Raynaud's monodromy factors through $\Z\hookrightarrow\Q$, we clearly have that the image of the composition (\ref{eq4.10}) lands in $\pi^\Z\Gm(S')/\Gm(S')\cong\pi^\Z$. Since the logarithmic monodromy pairing is induced by the pairing (\ref{eq4.7}), we are done.
\end{proof}

Under the assumption that Raynaud's monodromy factors through $\Z\hookrightarrow\Q$, we have Raynaud's factorizations $u_K=u^1_{K,\pi}+u^2_{K,\pi}$ and $u_K=u^1_K+u^2_K$, as well as the factorizations $u^{\mr{log}}=u_{1,\pi}^{\mr{log}}+u_{2,\pi}^{\mr{log}}$ and $u^{\mr{log}}=u_{1}^{\mr{log}}+u_{2}^{\mr{log}}$ from Proposition \ref{prop4.2} by Lemma \ref{lem4.1}. From the constructions of these decompositions, it is easy to check the following proposition.

\begin{prop}
The restrictions of the decompositions $u^{\mr{log}}=u_{1,\pi}^{\mr{log}}+u_{2,\pi}^{\mr{log}}$ and $u^{\mr{log}}=u_{1}^{\mr{log}}+u_{2}^{\mr{log}}$ from $S$ to $\Spec K$ give rise to Raynaud's decompositions $u_K=u^1_{K,\pi}+u^2_{K,\pi}$ and $u_K=u^1_K+u^2_K$ respectively.
\end{prop}

\section{Log finite group objects associated to k\'et log 1-motives}\label{sec5}

\subsection{Log finite group objects}
Let $S$ be a locally noetherian fs log scheme. Kato has developed a theory of log finite group objects, which is parallel to the theory of finite flat group schemes in the non-log world. The main references are \cite{kat4} and \cite{mad2}.

\begin{defn}\label{defn5.1}
The category $(\mr{fin}/S)_{\mr{c}}$ is the full subcategory of the category of sheaves of finite abelian groups over $(\mr{fs}/S)_{\mr{kfl}}$ consisting of objects which are representable by a classical finite flat group scheme over $S$. Here ``classical'' means that the log structure of the representing log scheme is the one induced from $S$. 

The category $(\mr{fin}/S)_{\mr{f}}$ is the full subcategory of the category of sheaves of finite abelian groups over $(\mr{fs}/S)_{\mr{kfl}}$ consisting of objects which are representable by a classical finite flat group scheme over a Kummer flat cover of $S$. For $F\in (\mr{fin}/S)_{\mr{f}}$, let $U\rightarrow S$ be a log flat cover of $S$ such that $F_U:=F\times_S U\in (\mr{fin}/U)_{\mr{c}}$. Then the rank of $F$ is defined to be  the rank of $F_U$ over $U$.

The category $(\mr{fin}/S)_{\mr{r}}$ is the full subcategory of $(\mr{fin}/S)_{\mr{f}}$ consisting of objects which are representable by a log scheme over $S$.

The category $(\mr{fin}/S)_{\mr{\acute{e}}}$ is the full subcategory of $(\mr{fin}/S)_{\mr{f}}$ consisting of objects $F$ such that there exists a Kummer \'etale cover $U$ of $S$ such that $F\times_SU\in (\mr{fin}/U)_{\mr{r}}$.

Let $F\in (\mr{fin}/S)_{\mr{f}}$. Then the Cartier dual of $F$ is the sheaf $F^*:=\mc{H}om_{S_{\mr{kfl}}}(F,\Gm)$. By the definition of $(\mr{fin}/S)_{\mr{f}}$, it is clear that $F^*\in (\mr{fin}/S)_{\mr{f}}$.

The category $(\mr{fin}/S)_{\mr{d}}$ is the full subcategory of $(\mr{fin}/S)_{\mr{r}}$ consisting of objects whose Cartier dual also lies in $(\mr{fin}/S)_{\mr{r}}$.
\end{defn}

\begin{prop}[Kato]\label{prop5.1}
The categories $(\mr{fin}/S)_{\mr{f}}$, $(\mr{fin}/S)_{\mr{\acute{e}}}$, $(\mr{fin}/S)_{\mr{r}}$, and  \\
$(\mr{fin}/S)_{\mr{d}}$ are closed under extensions in the category of sheaves of abelian groups on $(\mr{fs}/S)_{\mr{kfl}}$.
\end{prop}
\begin{proof}
See \cite[Prop. 2.3]{kat4}.
\end{proof}

\begin{defn}\label{defn5.2}
Let $p$ be a prime number. A \textbf{log $p$-divisible group} (resp. \textbf{k\'et log $p$-divisible group}, resp. \textbf{kfl log $p$-divisible group}) over $S$ is a sheaf of abelian groups $G$ on $(\mr{fs}/S)_{\mr{kfl}}$ satisfying:
\begin{enumerate}[(1)]
\item $G=\bigcup_{n\geq 0}G_n$ with $G_n:=\mr{ker}(p^n:G\rightarrow G)$;
\item $p:G\rightarrow G$ is surjective;
\item $G_n\in (\mr{fin}/S)_{\mr{r}}$ (resp. $G_n\in (\mr{fin}/S)_{\mr{\acute{e}}}$, resp. $G_n\in (\mr{fin}/S)_{\mr{f}}$) for any $n> 0$.
\end{enumerate}
We denote the category of log $p$-divisible groups (resp. k\'et log $p$-divisible groups, resp. kfl log $p$-divisible groups) over $S$ by $(\text{$p$-div}/S)^{\mr{log}}$ (resp. $(\text{$p$-div}/S)^{\mr{log}}_{\mr{\acute{e}}}$, resp. $(\text{$p$-div}/S)^{\mr{log}}_{\mr{f}}$). The full subcategory of $(\text{$p$-div}/S)^{\mr{log}}$ consisting of objects $G$ with $G_n\in (\mr{fin}/S)_{\mr{d}}$ for all $n>0$ will be denoted by $(\text{$p$-div}/S)^{\mr{log}}_{\mr{d}}$. A log $p$-divisible group $G$ with $G_n\in (\mr{fin}/S)_{\mr{c}}$ for all $n>0$ is clearly just a classical $p$-divisible group, and we denote the full subcategory of $(\text{$p$-div}/S)^{\mr{log}}_{\mr{d}}$ consisting of classical $p$-divisible groups by $(\text{$p$-div}/S)$.
\end{defn}

\begin{rmk}
For $G\in (\text{$p$-div}/S)^{\mr{log}}$ to lie in $(\text{$p$-div}/S)^{\mr{log}}_{\mr{d}}$, it is enough to require $G_1\in(\mr{fin}/S)_{\mr{d}}$. We explain this as follows. The short exact sequence $0\to G_1\to G_2\to G_1\to0$ gives an exact sequence
$$0\to G_1^*\to G_2^*\xrightarrow{\alpha} G_1^*\to \mc{E}xt^1_{S_{\mr{fl}}}(G_1,\Gm).$$
We claim that $\alpha$ is an epimorphism for the Kummer flat topology. To prove the claim, we may assume that $G_1$ is a classical finite flat group scheme. Then we have $\mc{E}xt^1_{S_{\mr{fl}}}(G_1,\Gm)=0$ by \cite[Exp. VIII, Prop. 3.3.1]{sga7-1} and part (1) of \cite[\href{https://stacks.math.columbia.edu/tag/0DDS}{Tag 0DDS}]{stacks-project}, whence a short exact sequence
$$0\to G_1^*\to G_2^*\to G_1^*\to0$$
of sheaves of abelian groups on $(\mr{fs}/S)_{\mr{kfl}}$.
Therefore $G_2^*\in(\mr{fin}/S)_{\mr{r}}$ by \cite[Prop. 2.3]{kat4}. Inductively we get $G_n^*\in(\mr{fin}/S)_{\mr{r}}$ for all $n>1$. 
\end{rmk}

\subsection{Log finite group objects associated to k\'et log 1-motives}
\begin{defn}
Let $S$ be an fs log scheme, $M^{\mr{log}}=[Y\xrightarrow{u}G_{\mr{log}}]$ a k\'et log 1-motive over $S$, and $n$ a positive integer. By Lemma \ref{lem2.2} and Corollary \ref{cor2.1}, we can regard $M^{\mr{log}}$ as a complex of sheaves on $(\mr{fs}/S)_{\mr{kfl}}$, and define 
$$T_{\Z/n\Z}(M^{\mr{log}}):=H^{-1}(M^{\mr{log}}\otimes_{\Z}^{\mr{L}}\Z/n\Z).$$
\end{defn}

\begin{prop}
Let $S$ be a locally noetherian fs log scheme, 
$$M^{\mr{log}}=[Y\xrightarrow{u}G_{\mr{log}}]$$
a k\'et log 1-motive over $S$, and $n$ a positive integer. Then we have the following.
\begin{enumerate}[(1)]
\item $T_{\Z/n\Z}(M^{\mr{log}})$ fits into the following exact sequence
$$0\rightarrow G_{\mr{log}}[n]\rightarrow T_{\Z/n\Z}(M^{\mr{log}})\rightarrow Y/nY\rightarrow0$$
of sheaves of abelian groups on $(\mr{fs}/S)_{\mr{kfl}}$, where $G_{\mr{log}}[n]$ denotes the $n$-torsion subgroup sheaf of $G_{\mr{log}}$.
\item $T_{\Z/n\Z}(M^{\mr{log}})\in(\mr{fin}/S)_{\mr{\acute{e}}}$.
\item Let $m$ be another positive integer. Then the map $T_{\Z/mn\Z}(M^{\mr{log}})\rightarrow T_{\Z/n\Z}(M^{\mr{log}})$ induced by $\Z/mn\Z\xrightarrow{m}\Z/n\Z$ is surjective.
\item If $M^{\mr{log}}$ is a log 1-motive, then $T_{\Z/n\Z}(M^{\mr{log}})\in(\mr{fin}/S)_{\mr{d}}$.
\end{enumerate}
\end{prop}
\begin{proof}
For part (1), by \cite[\S 3.1]{ray2}, it suffices to show that the multiplication by $n$ is injective on $Y$ and surjective on $G_{\mr{log}}$ for the Kummer flat topology. The injectivity of the map $Y\xrightarrow{n}Y$ is trivial. We are reduced to show the surjectivity of the map $G_{\mr{log}}\xrightarrow{n}G_{\mr{log}}$.  Without loss of generality, we may assume that $M^{\mr{log}}$ is a log 1-motive. Let $G$ be an extension of an abelian scheme $B$ by a torus $T$ over $S$. Consider the following commutative diagram
$$\xymatrix{
0\ar[r] &T_{\mr{log}}\ar[r]\ar[d]^n &G_{\mr{log}}\ar[r]\ar[d]^n &B\ar[r]\ar[d]^n &0   \\
0\ar[r] &T_{\mr{log}}\ar[r] &G_{\mr{log}}\ar[r] &B\ar[r] &0 
}$$
with exact rows. The multiplication by $n$ is clearly surjective on $B$, and the surjectivity of the multiplication by $n$ on $T_{\mr{log}}$ follows from the surjectivity of $\Gml\xrightarrow{n}\Gml$. It follows that $G_{\mr{log}}\xrightarrow{n}G_{\mr{log}}$ is surjective.

For part (2), we may still assume that $M^{\mr{log}}$ is a log 1-motive. We have a short exact sequence $0\rightarrow T_{\mr{log}}[n]\rightarrow G_{\mr{log}}[n]\rightarrow B[n]\rightarrow0$. Let $X$ be the character group of $T$. Then we get an exact sequence $$0\rightarrow T\rightarrow T_{\mr{log}}\rightarrow\mc{H}om_{S_{\mr{kfl}}}(X,\Gml/\Gm)\rightarrow0.$$
Since $\Gml/\Gm$ is torsion-free, we get $T[n]=T_{\mr{log}}[n]$. Then we get a short exact sequence $0\rightarrow T[n]\rightarrow G_{\mr{log}}[n]\rightarrow B[n]\rightarrow0$. Therefore $G_{\mr{log}}[n]\in(\mr{fin}/S)_{\mr{r}}$ by Proposition \ref{prop5.1}. Applying Proposition \ref{prop5.1} again to the short exact sequence 
$$0\rightarrow G_{\mr{log}}[n]\rightarrow T_{\Z/n\Z}(M^{\mr{log}})\rightarrow Y/nY\rightarrow0,$$
we get $T_{\Z/n\Z}(M^{\mr{log}})\in(\mr{fin}/S)_{\mr{r}}$.

Part (3) is clearly true for the two k\'et log 1-motives $[Y\rightarrow 0]$ and $[0\rightarrow G_{\mr{log}}]$. It follows that it also holds for $M^{\mr{log}}$.

At last, we prove part (4). By the proof of part (2) we get $T_{\Z/n\Z}(M^{\mr{log}})\in(\mr{fin}/S)_{\mr{r}}$. Similarly, we have $T_{\Z/n\Z}(M^{\mr{log}})^{*}=T_{\Z/n\Z}((M^{\mr{log}})^{\vee})\in(\mr{fin}/S)_{\mr{r}}$, where $(M^{\mr{log}})^{\vee}$ denotes the dual of the log 1-motive $M^{\mr{log}}$. It follows that $T_{\Z/n\Z}(M^{\mr{log}})\in(\mr{fin}/S)_{\mr{d}}$.
\end{proof}

\begin{defn}
Let $S$ be a locally noetherian fs log scheme, 
$$M^{\mr{log}}=[Y\xrightarrow{u}G_{\mr{log}}]$$
a k\'et log 1-motive over $S$, and $p$ a prime number. The \textbf{k\'et log $p$-divisible group of $M^{\mr{log}}$} is defined to be $M^{\mr{log}}[p^{\infty}]:=\bigcup_n T_{\Z/p^n\Z}(M^{\mr{log}})$. 
\end{defn}

\subsection{Extending finite group schemes associated to tamely ramified strict 1-motives}

\begin{thm}\label{thm5.1}
Let the notation and the assumptions be as in Theorem \ref{thm3.1}, and let $n$ be a positive integer. Then $T_{\Z/n\Z}(M^{\mr{log}})$ lies in $(\mr{fin}/S)_{\mr{\acute{e}}}$, and it extends the finite group scheme $T_{\Z/n\Z}(M_K)$ over $K$ to $S$.
\end{thm}
\begin{proof}
Since $T_{\Z/n\Z}(M^{\mr{log}})\times_SS'=T_{\Z/n\Z}(M^{\mr{log}}\times_SS')\in(\mr{fin}/S)_{\mr{r}}$ and $S'$ is a Kummer \'etale cover of $S$, we get $T_{\Z/n\Z}(M^{\mr{log}})\in(\mr{fin}/S)_{\mr{\acute{e}}}$. Since $M^{\mr{log}}\times_S\Spec K=M_K$, we get $T_{\Z/n\Z}(M^{\mr{log}})\times_S\Spec K=T_{\Z/n\Z}(M_K)$.
\end{proof}

The following theorem is stated in \cite[\S 4.3]{kat4} without proof. Here we present a proof.

\begin{thm}[Kato]\label{thm5.2}
Let $K$ be a complete discrete valuation field with ring of integers $R$, $p$ a prime number, and $A_K$ a tamely ramified abelian variety over $K$. We endow $S:=\Spec R$ with the canonical log structure. Then the $p$-divisible group $A_K[p^{\infty}]$ of $A_K$ extends to an object of $(p\mr{\textnormal{-}div}/S)^{\mr{log}}_{\mr{\acute{e}}}$. It extends to an object of $(p\mr{\textnormal{-}div}/S)^{\mr{log}}_{\mr{d}}$ if any of the following two conditions is satisfied.
\begin{enumerate}[(1)]
\item $A_K$ has semi-stable reduction.
\item $p$ is invertible in $R$.
\end{enumerate} 
\end{thm}
\begin{proof}
By \cite[\S 4.2]{ray2}, there exists a tamely ramified strict 1-motive $M_K=[Y_K\xrightarrow{u_K}G_K]$ such that $M_K[p^{\infty}]=A_K[p^{\infty}]$, and both $Y_K$ and $G_K$ have good reduction if $A_K$ has semi-stable reduction. By Theorem \ref{thm3.1}, $M_K$ extends to a k\'et log 1-motive $M^{\mr{log}}=[Y\xrightarrow{u^{\mr{log}}}G_{\mr{log}}]$. Then $M_K[p^{\infty}]$ extends to $M^{\mr{log}}[p^{\infty}]\in (\text{$p$-div}/S)^{\mr{log}}_{\mr{\acute{e}}}$ by Theorem \ref{thm5.1}. 

If $A_K$ has semi-stable reduction, then both $Y_K$ and $G_K$ have good reduction. Therefore the k\'et log 1-motive $M^{\mr{log}}$ is actually a log 1-motive over $S$ by Corollary \ref{cor3.1}. It follows that $M^{\mr{log}}[p^{\infty}]\in (\text{$p$-div}/S)^{\mr{log}}_{\mr{d}}$. 

If $p$ is invertible in $R$, then the object $T_{\Z/p^n\Z}(M^{\mr{log}})\in (\mr{fin}/S)_{\mr{\acute{e}}}$ actually lies in $(\mr{fin}/S)_{\mr{d}}$ by \cite[Prop. 2.1]{kat4}. It follows that $M^{\mr{log}}[p^{\infty}]\in (\text{$p$-div}/S)^{\mr{log}}_{\mr{d}}$.
\end{proof}

\section*{Acknowledgement}
In an email, Professor Chikara Nakayama informed the author that Professor Kazuya Kato thought it plausible that every abelian variety (not necessarily with semi-stable reduction) on a complete discrete valuation field extends uniquely to a Kummer log flat log abelian variety on the corresponding discrete valuation ring. This work is partly motivated by that piece of information. It is also motivated by Theorem \ref{thm5.2} which is taken from \cite[\S 4.3]{kat4}. The author thanks Professor Chikara Nakayama for his generosity. The author would also thank Professor Ulrich G\"ortz for very helpful discussions concerning taking quotient for equivalence relations, as well as for his support during the past few years.

The author thanks the anonymous referee for her or his corrections. More importantly, he would like to thank the referee for her or his questions and comments which not just motivated the author to give a more systematic treatment to the results in Section \ref{sec3}, but also directed him to possible research problems in future.

This work has been partially supported by SFB/TR 45 ``Periods, moduli spaces and arithmetic of algebraic varieties''.

\bibliographystyle{alpha}
\bibliography{bib}

\end{document}